\documentclass[]{article}
\usepackage{amsmath,amssymb,amsthm}
\usepackage{enumerate}
\usepackage{graphicx}
\usepackage{enumitem}
\usepackage{pdfpages}
\usepackage[margin=3.9cm]{geometry}
\usepackage[margin=6ex,font=small]{caption}
\usepackage{lipsum}
\usepackage{tikz}
\usetikzlibrary{decorations.markings}
\usepackage{multicol}
\usepackage{caption}
\usepackage[UKenglish]{babel}

\makeatletter
\renewcommand*{\@fnsymbol}[1]{\@alph{#1}}
\makeatother

\title{C*-algebras of higher-rank graphs from groups acting on buildings, and explicit computation of their K-theory}
\author{
	Sam A. Mutter,
	Aura-Cristiana Radu, 
	Alina Vdovina
}
\date{\today}

\theoremstyle{plain}
\newtheorem{thm}{Theorem}[section]
\newtheorem{prop}[thm]{Proposition}
\newtheorem{conj}[thm]{Conjecture}
\newtheorem{lem}[thm]{Lemma}
\newtheorem{corl}[thm]{Corollary}

\theoremstyle{definition}
\newtheorem{defn}[thm]{Definition}

\newtheorem{ex}[thm]{Example}

\newtheorem{rem}[thm]{Remark}

\newcommand{\inv}{^{-1}}
\DeclareMathOperator{\coker}{coker}
\DeclareMathOperator{\rk}{rk}
\DeclareMathOperator{\Ob}{Ob}
\DeclareMathOperator{\Hom}{Hom}
\newcommand{\cst}{C^{\star}}
\DeclareMathOperator{\im}{im}
\DeclareMathOperator{\cat}{CAT}
\DeclareMathOperator{\lk}{lk}
\DeclareMathOperator{\Aut}{Aut}
\DeclareMathOperator{\id}{id}

\newcommand{\del}{\partial}
\newcommand{\bmu}{\boldsymbol\mu}
\newcommand{\bnu}{\boldsymbol\nu}

\begin{document}

\maketitle

\begin{abstract}
	We unite elements of category theory, K-theory, and geometric group theory, by defining a class of groups called $k$-cube groups, which act freely and transitively on the product of $k$ trees, for arbitrary $k$. The quotient of this action on the product of trees defines a $k$-dimensional cube complex, which induces a higher-rank graph. We make deductions about the K-theory of the corresponding rank-$k$ graph $\cst$-algebras, and give examples of $k$-cube groups and their K-theory. These are among the first explicit computations of K-theory for an infinite family of rank-$k$ graphs for $k\geq 3$, which is not a direct consequence of the K\"{u}nneth Theorem for tensor products.
\end{abstract}

\section{Introduction}

A \textit{rank-$k$ graph} is a combinatorial object defined by Kumjian and Pask in \cite{KumPas2000}, with motivation from Robertson and Steger \cite{RobSte1999}. To each rank-$k$ graph can be assigned a \textit{rank-$k$ graph algebra}, being the universal $\cst$-algebra generated by a set of partial isometries, and deductions about the K-theory of these algebras can be made courtesy of the \textit{spectral sequences} exhibited by Evans in \cite{Eva2008}. In general, it is difficult to calculate K-theory for higher-rank graphs. In this article we construct an infinite family of rank-$k$ graphs, for arbitrary $k$, and remark on the K-theory when $k=3,4,5$. 

Robertson and Steger laid the foundations for these $k$-dimensional generalisations of \textit{graph algebras}, which are $\cst$-algebras built with data obtained from $k$-many directed graphs on the same vertex set. The K-theory of these algebras when $k=2$ has been detailed in \cite{RobSte2001}, \cite{KimRob2002}, and until now, only a small selection of examples had been investigated in this way. Here, we further develop Robertson and Steger's methods, using Kumjian and Pask's category-theoretical language, to present an infinite family of examples arising from groups which act on the product of $k$-many regular trees. We outline the process now.

A \textit{VH-structure} on a square complex $\mathcal{M}$ is a partition of its directed edge set $E(\mathcal{M}) = E_V \sqcup E_H$ such that the link at each vertex $x$ of $\mathcal{M}$ is the complete bipartite graph induced by the partition. Such objects were pioneered by \cite{Wis1996} and further studied in \cite{BurMoz1997}, \cite{BurMoz2000}, \cite{KimRob2002}. It was shown, amongst other places, in \cite{BalBri1995}, that the universal cover of a square complex is a product of two trees if and only if the link at each vertex is a complete bipartite graph. Let $T_1$, $T_2$ be regular trees of constant valencies $m$, $n$, respectively, and let $\Gamma$ be a \textit{lattice} in $T_1 \times T_2$, that is, a group which acts discretely and cocompactly on $\Aut(T_1) \times \Aut(T_2)$, respecting the structure of cube complexes and such that $\Gamma \setminus (T_1 \times T_2)$ is a finite square complex. From \cite{BurMoz2000} we learn that $\Gamma$ corresponds uniquely to a square complex with VH-structure of partition size $(m,n)$, up to isomorphism; indeed, $\Gamma \setminus (T_1 \times T_2)$ is such a complex.

Let $\Gamma$ be a group, and let $E_1, E_2 \subseteq \Gamma$ be finite subsets closed under inverses. We say that the pair $(E_1,E_2)$ is a \textit{VH-structure} on $\Gamma$ if $E_1 \cup E_2$ generates $\Gamma$, and the product sets $E_1E_2$ and $E_2E_1$ are equal with size $|E_1| \cdot |E_2|$, and without $2$-torsion. We may define a \textit{BM-group} (named for Burger and Mozes, and developed by them and Wise) as a group which admits a VH-structure, as in \cite{BurMoz1997}, \cite{Wis1996}. A BM-group acts freely and transitively on a product of two trees, and yields a square complex with one vertex and a VH-structure in a natural way \cite{BurMoz2000}. We may therefore freely interchange geometric and algebraic terminology.

Inspired by this, Vdovina in \cite{VdoPrep} defined a $k$-dimensional generalisation: a partition into $k$ subsets $E_1 \sqcup \cdots \sqcup E_k$ of the directed edge set of a $k$-dimensional cube complex $\mathcal{M}$ which uniquely corresponds to a lattice in the product of $k$ trees of constant valencies $|E_1|, \ldots , |E_k|$. We call this an \textit{adjacency structure} for $\mathcal{M}$. Likewise, we define a $k$\textit{-cube group} to be a group $\Gamma$ which admits $k$ subsets $E_1, \ldots , E_k$ which satisfy a compatibility condition, whose union generates $\Gamma$ and which, taken pairwise, have the same properties as above. A $k$-cube group acts freely and transitively on a product of $k$ trees---this is a rank-$k$ affine building $\Delta$ which is thick whenever $|E_i| > 2$ for all $i$. The quotient of this action is a $k$-dimensional cube complex with one vertex and endowed with a $k$-adjacency structure. We identify the $k$-dimensional cells (chambers) of $\Delta$ with the tuple of elements of $\Gamma$ which label their edges; in this way we define the notion of a $k$\textit{-cube in the group} $\Gamma$.  

In Section \ref{S:k-cubes}, we construct \textit{adjacency functions} on the sets of $k$-cubes of $\Gamma$, defining two $k$-cubes to be adjacent if the cells they define in $\Delta$ are adjacent in a certain way. This generalises the $2$-dimensional shift system explored in \cite{KimRob2002}, where two squares were considered adjacent if they could be stacked against one another horizontally or vertically. In $k$ dimensions, there are $k$ `directions' in which $k$-cubes can be stacked, and hence we define $k$-many adjacency functions. We show in Proposition \ref{prop:ufp} that the adjacency functions satisfy a Unique Common Extension Property in the following sense: firstly, consider a $2 \times 2 \times 2$ arrangement of $k$-cubes. Then, given an initial $k$-cube and three $k$-cubes adjacent to it in three mutually-orthogonal directions, we can uniquely find four more $k$-cubes which fill in the $2 \times 2 \times 2$ structure (Figure \ref{fig:ufp}). In \cite{RobSte1999}, Roberston and Steger show that this $3$-dimensional commutativity of $k$-cubes is enough to imply unique common extensions in all dimensions up to $k$. As such, we are able to conclude in Section \ref{S:higher-rank} that a $k$-cube group $\Gamma$ induces a rank-$k$ graph $\mathcal{G}(\Gamma)$.

In Theorem \ref{thm:determined_by_K}, we show that the K-theory of our rank-$k$ graph $\cst$-algebras determines the $\cst$-algebras uniquely, up to isomorphism.

We use a technique of \cite{RunStiVdo2019} to build examples of $k$-cube groups, and uncover enough about their K-theory to be able to distinguish their induced rank-$k$ graph $\cst$-algebras. We note that a $k$-cube group $\Gamma$ is an amalgamated product of $(k-1)$-cube groups, and the induced K-theory has no immediately-discernible relation to that of the $(k-1)$-cube groups which $\Gamma$ contains. Until now, there are very few explicit constructions for infinite families of higher-rank graph algebras for which the K-theory is (even partially) known (see, for example, \cite{KPS2008}).

In the final section, we use Evans' K-theory formulas and our own corollaries to study some rank-$k$ graphs arising from cube complexes in a second way, namely as double covers of the cube complexes constructed in Section \ref{S:k-cubes}. This work follows from that of \cite{LawSimVdoPrep}.

\section{$k$-cube groups}\label{S:k-cubes}

For some finite $n \geq 2$, define $T(n)$ to be the regular tree of degree $n$. We may simply write $T$ if the degree is not important. 

Let $T_1, \ldots , T_k$ be regular trees, and consider the product $T_1 \times \cdots \times T_k$. This defines a $k$-dimensional cube complex $\Delta$, which is an affine building of rank $k$.

Recall that the \textbf{link} at a vertex $x$ of a $k$-dimensional cell complex $G$ is the $(k-1)$-dimensional cell complex $\lk_x(G)$ obtained as the intersection of $G$ with a small $2$-sphere centred at $x$.

\begin{prop}\label{prop:clique}
	Let $\mathcal{M}$ be a $k$-dimensional cube complex. The universal cover of $\mathcal{M}$ is a product of $k$ trees $\tilde{\mathcal{M}} = T_1 \times \cdots \times T_k$ if and only if the link at each vertex of $\mathcal{M}$ is a clique complex of a complete $k$-partite graph.
\end{prop}

\begin{proof}
	This proposition is a generalisation of Theorem 10.2 in \cite{BriWis1999}. Observe that if the link $\lk_x(\mathcal{M})$ at a vertex $x$ of $\mathcal{M}$ is such a clique complex, then $\lk_x(\mathcal{M})$ is a $(k-1)$-dimensional complex such that every cycle has length at least $k$. Hence $\lk_x(\mathcal{M})$ is $\cat(1)$, and so by the \textit{Gromov Link Condition} \cite[\S 4.2]{Gro1988}, $\mathcal{M}$ must be $\cat(0)$. The result then follows from a relatively straightforward adaptation to Theorem 4.3 in \cite{BriWis1999}. 
\end{proof}

The following definitions generalise objects from \cite{KimRob2002} and \cite{Wis1996}.

\begin{defn}
	Let $\mathcal{M}$ be a $k$-dimensional cube complex with vertex set $V$ and edge set $E$. Assign to each edge $u \in E$ an orientation (these are arbitrary but will remain fixed once chosen) and write $u\inv$ to denote the oriented edge with opposite orientation to $u$. For each $x \in V$, write $E(x)$ for the set of oriented edges originating at $x$. Suppose that we have a partition $E = E_1 \sqcup \cdots \sqcup E_k$ such that $u\inv \in E_i$ whenever $u \in E_i$, and suppose that for each vertex $x \in V$, the $1$-skeleton of the link at $x$ is the complete $k$-partite graph with vertices according to the partition $E(x) = E(x)_1 \sqcup \cdots \sqcup E(x)_k$. We say that $E_1,\ldots , E_k$ form an \textbf{adjacency structure} for $\mathcal{M}$. 
	
	Note that not every $k$-dimensional cube complex $\mathcal{M}$ admits an adjacency structure, but the criterion that the link at each vertex of $\mathcal{M}$ be a clique complex of a complete $k$-partite graph is sufficient to ensure that it does (c.f. Remark \ref{rem:ufp}(ii)).
\end{defn}

\begin{figure}[ht]
	\begin{center}
		\includegraphics[scale=0.9]{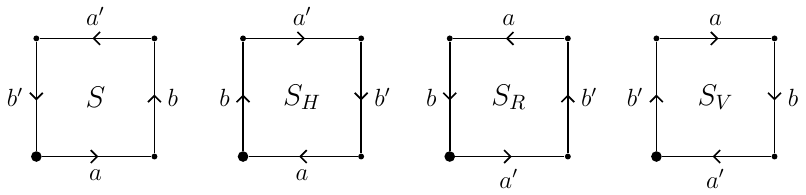}
	\end{center}
	\caption[.]{Given a pointed square $S = [a,b,a',b'] \in \mathcal{S}_2$, we write $S_H$, $S_R$, $S_V$ to denote the three pointed squares in $\mathcal{S}_2$ depicted above (compare with (\ref{eq:squares})). Geometrically, these squares are all in the same orbit under reflection in the horizontal and/or vertical directions, meaning they represent the same element in $\mathcal{S}_2'$. Figure \ref{fig:cube_symmetries} extends this notion to cubes.}\label{fig:square_symmetries}
\end{figure}

\begin{defn}\label{def:cube_complex}
	Let $\mathcal{M}$ be a $k$-dimensional cube complex with vertex set $V$, edge set $E$, and adjacency structure $E_1,\ldots , E_k$. Write $\mathcal{S}_2' = \mathcal{S}_2'(\mathcal{M})$ for the set of geometric squares of which $\mathcal{M}$ consists. We write elements of $\mathcal{S}_2'$ as ordered $4$-tuples of their oriented edge labels $(a,b,a',b')$ for $a,a' \in E_i$, $b,b' \in E_j$, where the map $e \mapsto e\inv$ reverses the orientation of the edge. We use square (``pointy") brackets if we wish to emphasise that a square is labelled according to some predetermined orientation and starting from some basepoint (``pointed"), and write $\mathcal{S}_2 = \mathcal{S}_2(\mathcal{M})$ to denote the set of all pointed squares. For each square $S = [a,b,a',b'] \in \mathcal{S}_2$, write:
	\begin{equation} \label{eq:squares}
		S_H := \big[ a\inv, (b')\inv, (a')\inv, b\inv \big], \, S_R :=
		\big[ a',b',a,b \big], \, S_V := \big[ (a')\inv, b\inv, a\inv, (b')\inv \big];
	\end{equation}
	geometrically these can be interpreted as the pointed squares which lie in the same orbit of $S$ under the actions of reflection in the $a$ direction, rotation by $\pi$, and reflection in the $b$ direction (Figure \ref{fig:square_symmetries}). Then 
	\[
	\mathcal{S}_2 = \lbrace S,S_H,S_R,S_V \mid S \in \mathcal{S}_2' \rbrace.
	\]
	Write $F(p,q) := \lbrace [a,b,a',b'] \in \mathcal{S}_2 \mid a,a' \in E_p\text{, and } b,b' \in E_q \rbrace$, and identify $F(p,q)$ with $F(q,p)$ via the map  $\varphi : [a,b,a',b'] \mapsto \big[ (b')\inv , (a')\inv , b\inv , a\inv\big]$.
	
	Similarly, we write $\mathcal{S}_3' = \mathcal{S}_3'(\mathcal{M})$ for the set of geometric cubes which $\mathcal{M}$ comprises, and we denote elements of $\mathcal{S}_3'$ by ordered $6$-tuples of their faces $(A,B,C,A',B',C')$ for $A,A' \in F_{ij}$, $B,B' \in F_{il}$, and $C,C' \in F_{jl}$. As above, we use square brackets to indicate that a cube is pointed and oriented, and for each cube $S = [ A,B,C,A',B',C'] \in \mathcal{S}_3$, we write:
	\begin{equation}\label{eq:cube_symmetries}
		\begin{array}{rclrcl}
			S_H &:=& \big[ A_H,B_H,C_H',A_H',B_H',C_H \big],
			& S_{HI} &:=& \big[ A',B_R,C_R',A,B_R',C_R \big], \\
			S_R &:=& \big[ A_R,B',C',A_R',B,C \big],
			& S_{RI} &:=& \big[ A_V',B_V',C_V',A_V,B_V,C_V \big], \\
			S_V &:=& \big[ A_V,B_H',C_H,A_V',B_H,C_H' \big],
			& S_{VI} &:=& \big[ A_R',B_R',C_R,A_R,B_R,C_R' \big], \\
			S_I &:=& \big[ A_H',B_V,C_V,A_H,B_V',C_V' \big]. & & &
		\end{array}
	\end{equation}
	These are the cubes $[X_1,\ldots ,X_6]$ which belong to the same orbit as $[A,B,C,A',B',C']$ under action by the symmetry group of the cube, with the property that if $A \in F(i,j)$, then $X_1 \in F(i,j)$ (Figure \ref{fig:cube_symmetries}). Write $\mathcal{S}_3$ for the set which comprises each $S \in \mathcal{S}_3'$ and all of the corresponding pointed cubes above. Write
	\begin{multline*}
		F(p,q,r) := \lbrace [A,B,C,A',B',C'] \in \mathcal{S}_3 \mid A,A' \in F(p,q) \\ B,B' \in F(p,r)\text{, and }C,C' \in F(q,r)\rbrace,
	\end{multline*}
	and identify $F(a,b,c)$ with $F(a,c,b)$ via the map 
	\[(A,B,C,A',B',C') \longmapsto \big(\varphi(A),C_H',B_H',\varphi(A'),C_H,B_H\big),
	\]
	where $\varphi([a,b,a',b']) = [(b')\inv, (a')\inv, b\inv, a\inv]$. Likewise we are able to identify $F(a,b,c)$ with each of the sets $F(\sigma(a,b,c))$, for each permutation $\sigma$.  
	
	For $3 \leq n \leq k$, we inductively define the sets $\mathcal{S}_n' = \mathcal{S}_n'(\mathcal{M})$ of geometric $n$-cubes of which $\mathcal{M}$ consists, and we write elements of $\mathcal{S}_n'$ as ordered $(2n)$-tuples of their faces (the incident elements of $\mathcal{S}_{n-1}$). Here, we define $\mathcal{S}_n$ as the set of all $(2n)$-cubes $(A_1,\ldots ,A_{2n})$ which belong to the orbit of some $(A_1^1,\ldots , A_1^n,A_2^1,\ldots ,A_2^n) \in \mathcal{S}_n'$ under the action of the group of reflections of the $(2n)$-cube; analogously to above, we write 
	\[
	F(p_1, \ldots , p_n) := \big\lbrace \big[A_1^1,\ldots , A_1^n,A_2^1,\ldots ,A_2^n \big] \in \mathcal{S}_n \bigm\vert A_1^i,A_2^i \in F\big( p_1, \ldots , \hat{p}_{n-i+1}, \ldots , p_n \big) \big\rbrace ,
	\]
	and we can identify the sets $F(\sigma(p_1,\ldots , p_n))$ for each permutation $\sigma$.
\end{defn}

\begin{figure}[ht]
	\begin{center}
		\includegraphics[scale=0.9]{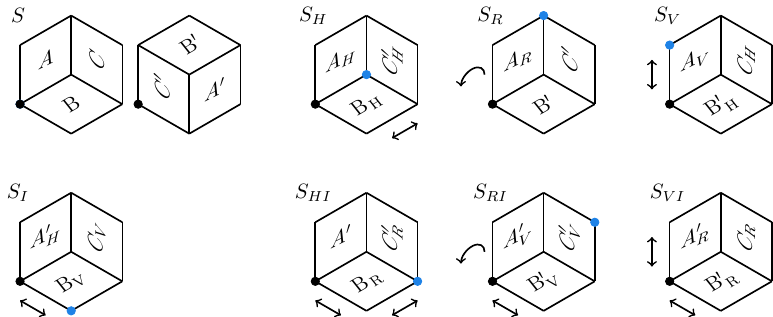}
	\end{center}
	\caption[.]{For a pointed, oriented $3$-cube $S = [A,B,C,A',B',C'] \in \mathcal{S}_3$, the seven corresponding cubes from (\ref{eq:cube_symmetries}) are defined by reflecting and rotating $S$ according to the arrows above. The transformations map the original basepoint to a new vertex (blue), but the new cubes are given the same basepoint and orientation as $S$ (black).}\label{fig:cube_symmetries}
\end{figure}

\begin{rem}
	To recap, given a $k$-dimensional cube complex $\mathcal{M}$ with an adjacency structure, we assign an (arbitrary but fixed) orientation and basepoint to each $n$-dimensional cell of $\mathcal{M}$. We write $\mathcal{S}_n'$ to denote the geometric $n$-dimensional cells of $\mathcal{M}$---each element of $\mathcal{S}_n'$ gives rise to $2^n$ pointed and oriented $n$-cells (e.g. when $n=1$ we called these $u$ and $u\inv$, and when $n=2$ we called them $S$, $S_H$, $S_R$ and $S_V$). Giving each cell a basepoint like this will help us determine which cells are \textit{adjacent} in Definition \ref{def:adjacency_functions}.
\end{rem}

\begin{defn}\label{def:BM}
	Let $k \geq 2$, and let $E_1,\ldots , E_k$ be finite sets of respective even cardinalities $2m_1,\ldots, 2m_k$, with each $m_i \geq 2$. Suppose that each set $E_i$ is endowed with a fixed-point-free involution, denoted $a \mapsto a\inv$. For each $i, j$ with $i \neq j$, write $F(i,j) := E_i \times E_j \times E_i \times E_j$, let $R \subseteq \bigsqcup_{i \neq j} F(i,j)$, and define the group 
	\[
	\Gamma := \langle E_1 \sqcup \cdots \sqcup E_k \mid aba'b' = 1 \text{ whenever } (a,b,a',b') \in R \rangle.
	\]
	Suppose firstly that $k=2$. We call $\Gamma$ a \textbf{BM-group} if $R$ has the following properties:
	\begin{itemize}
		\item[\textbf{C1}] For each element $(a,b,a',b') \in R$, each of $(a\inv, (b')\inv, (a')\inv, b\inv)$, $(a',b',a,b)$, and $((a')\inv, b\inv, a\inv, (b')\inv)$ is also in $R$, and all four $4$-tuples are distinct.
		\item[\textbf{C2}] Each of the projections of $R$ onto the subproducts of the form $E_i \times E_j$ or $E_j \times E_i$, for all $i \neq j$, is bijective.
	\end{itemize}
	These groups were developed extensively in \cite{BurMoz1997}, \cite{KimRob2002} and \cite{Wis1996}. In order to generalise BM-groups to $k\geq 3$, we require the construction of subsets $\mathcal{S}_n \subseteq \prod_{\alpha=1}^{2n} R$ for each $n \in \lbrace 2,\ldots , k\rbrace$. We begin by illustrating $\mathcal{S}_3$. Compare \textbf{C1} with (\ref{eq:squares}); we will have in mind that each element of $R$ can label a pointed square. Then elements of $\mathcal{S}_n$ will be $2n$-tuples of such squares which can be arranged to form $n$-dimensional cubes, as in Definition \ref{def:cube_complex}.
	
	Suppose now that $k \geq 3$, and that $R$ has properties \textbf{C1} and \textbf{C2}. Write $R(p,q) := R \cap F(p,q)$, and fix $(a_1,b_1,a_2,b_2) \in R(p,q)$ and $(a_1,c_1,a_3,c_2) \in R(p,r)$. 
	
	Also suppose that we can find some \textit{unique} elements $a_4,b_3,b_4,c_3,c_4 \in \bigcup_l E_l$ such that $\big(b_1,c_3\inv,b_4,c_1\inv\big)$, $\big(b_2,c_2\inv,b_3,c_4\big) \in R(q,r)$, $\big(a_2,c_4\inv,a_4,c_3\big) \in R(p,r)$, and $(a_3,b_3,a_4,b_4) \in R(p,q)$. Equivalently, suppose that the same is true if we are given $(a_1,b_1,a_2,b_2) \in R(p,q)$ and $\big(b_1,c_3\inv,b_4,c_1\inv\big) \in R(q,r)$. Geometrically, we can view each $4$-tuple as a square, such that each of $a_1, \ldots , a_4$, $b_1, \ldots , b_4$, $c_1, \ldots , c_4$ labels the edges of a cube. Write $\mathcal{S}_3$ for the set of $6$-tuples of elements of $R$ which correspond to the faces of all such cubes, pointed and oriented according to some predetermined orientation.
	
	We extend the notion of $\mathcal{S}_3$ to that of $\mathcal{S}_n$ as follows. Suppose that $k \geq n$, and fix $p \in \lbrace 1,\ldots , k\rbrace$. Let $J \subseteq (\lbrace 1,\ldots , k\rbrace \setminus \lbrace p\rbrace )$ be some subset of cardinality $(n-1)$, and let $L \subseteq J$ denote a subset of cardinality $|L| \geq 0$.
	
	For each $u \in E_p$ and each $j \in J$, we fix an element $\big(u, v_j, u^j, w_j\big) \in R(p,j)$. We assume that, for each $L \subseteq J$ with $0 \leq |L| \leq n-2$, we can find \textit{unique} elements $u^L \in E_p$ and $v_i^L, w_i^L \in E_i$ such that:
	\begin{enumerate}[label=(\alph*)]
		\item $\Big(u^{L}, v_{j}^{L}, u^{L \cup \lbrace j\rbrace}, w_{j}^{L}\Big) \in R(p, j)$, whenever $j \in J$ and $j \notin L$.
		\item $\Big( \big(v_{i}^{L}\big)\inv , v_{j}^{L} , \Big(w_{i}^{L\cup \lbrace j\rbrace}\Big)\inv , w_{j}^{L\cup \lbrace i\rbrace}\Big) \in R(i,j)$, for all $i,j \in J$ with $i \neq j$ and $i,j \notin L$.
		\item $\Big( w_{i}^{L}, \big( w_{j}^{L}\big)\inv , v_{i}^{L\cup \lbrace j\rbrace}, \Big( v_{j}^{L\cup \lbrace i\rbrace}\Big)\inv\Big) \in R(i,j)$, for all $i,j \in J$ with $i \neq j$ and $i,j \notin L$.
	\end{enumerate} 
	We write $\square(u,v_J,u^J,w_J)$ for the $2n$-tuple of $R$ comprising the initial choices $(u,v_j,u^j,w_j)$ and the elements above which they uniquely determine, listed according to some predetermined order. We write $\mathcal{S}_n = \mathcal{S}_n(\Gamma)$ for the set of all $2n$-tuples $\square(u,v_J,u^J,w_J)$. Elements of $\mathcal{S}_n$ can be regarded as pointed, oriented $n$-cubes whose $2$-faces are labelled by elements of $R$ (Figure \ref{fig:k-dim_cubes}). We call elements of $\mathcal{S}_n$ \textbf{pointed} $n$\textbf{-cubes}.
	
	Note that we may identify each $F(i,j)$ with $F(j,i)$ (and hence $R(i,j)$ with $R(j,i)$) via the isomorphism $(a,b,a',b') \mapsto \big((b')\inv),(a')\inv,b\inv,a\inv\big)$, such that we need only consider only those $i,j$ with $i < j$. We refrained from doing so immediately in order to simplify some of the above notation. 
	
	Summarising the above process, if $k\geq 3$, we call $\Gamma$ a $k$\textbf{-cube group} whenever, in addition to conditions \textbf{C1} and \textbf{C2} above, $R$ also satisfies:
	\begin{itemize}
		\item[\textbf{C3}] Fix $p \in \lbrace 1, \ldots , k \rbrace$. Then for each set of $(n-1)$ elements $\big(u, v_j, u^j, w_j\big) \in R(p,j)$, where $j \in J \subseteq ( \lbrace 1, \ldots , k \rbrace \setminus \lbrace p \rbrace)$ and $|J| = n-1$, then for each $L \subseteq J$ with $0 \leq |L| \leq n-2$ we can find \textit{unique} elements $u^L,v_i^L, w_i^L \in E_i$ which satisfy (a)--(c) above, and hence we are able to construct subsets $\mathcal{S}_n \subseteq \prod_{\alpha=1}^{2n} R$ for each $n \in \lbrace 2,\ldots , k\rbrace$.
	\end{itemize}
\end{defn}

\begin{rem}
	In \cite{KhaVdo2018}, an alternative condition was given on $R$, which is equivalent to \textbf{C1} and \textbf{C2}:
	\begin{itemize}
		\item[\textbf{C1'}] The product sets $E_i E_j$ and $E_j E_i$ are equal, contain no $2$-torsion, and have cardinality $|E_i E_j | = |E_i| \cdot |E_j| = 4m_im_j$.
	\end{itemize}
	Since we are explicitly constructing $k$-cube complexes, we mainly rely on properties \textbf{C1} and \textbf{C2} in this paper.
\end{rem}

\begin{figure}
	\begin{center}
		\includegraphics[scale=0.9]{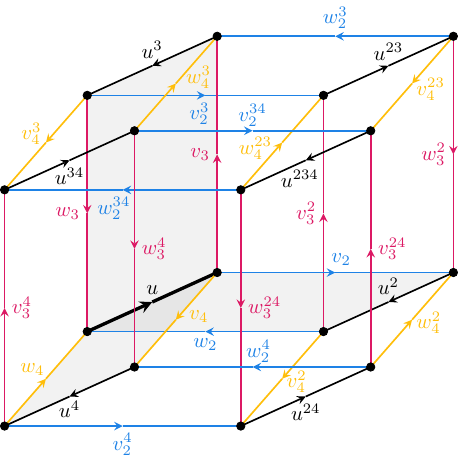}
	\end{center}
	\captionsetup{singlelinecheck=off}
	\caption[.]{Let $k \geq 4$. Above is depicted a pointed $4$-cube in $\mathcal{S}_4$, for some $k$-cube group with adjacency structure $E_1,\ldots , E_k$. Let $u^L$ be elements of $E_1$, and $v_i^L,w_i^L \in E_i$. Fix three mutually-adjacent squares: elements of $F(1,j)$ labelled $(u,v_j,u^j,w_j)$, for $j \in \lbrace 2,3,4\rbrace$. Then each of the remaining $u^L,v_i^L,w_i^L$ is uniquely-determined, such that they label the edges of the $4$-cube above. We have condensed the notation for the sets $L$ for clarity.}\label{fig:k-dim_cubes}
\end{figure}

\begin{lem}\label{lem:k-1-cube}
	Let $\Gamma = \langle X \mid R\rangle$ be a $k$-cube group with adjacency structure $E_1,\ldots , E_k$, and let $\Gamma'(p) = \langle X'(p) \mid R'(p)\rangle \subset \Gamma$ be the subgroup obtained by removing all of the elements of some set $E_p$ from the generating set $X$. Then $\Gamma'(p)$ is a $(k-1)$-cube group, with adjacency structure $E_1,\ldots , \hat{E}_p,\ldots , E_k$.
	
	By induction, we can form a $(k-m)$-cube subgroup by removing all elements of $m$ sets $E_{p_1}, \ldots , E_{p_m}$ from the generating set $X$. We denote such a group by $\Gamma'(p_1,\ldots , p_m)$.
\end{lem}

\begin{proof}
	It follows immediately from the definition of \textbf{C3} that if $R$ satisfies \textbf{C3} for some given $k$, then the set $R \setminus R(p,j)$ satisfies \textbf{C3} for dimension $k-1$. Iteration of this process demonstrates that each set $R(p_1, \ldots , p_m)$ satisfies \textbf{C3} for each $m \leq k$, which is enough to prove the lemma.
\end{proof}

\begin{prop}\label{prop:amalgam}
	Let $\Gamma$ be a $k$-cube group with adjacency structure $E_1,\ldots , E_k$, and write $\Gamma'(p_1,\ldots , p_m) = \langle X'(p_1,\ldots , p_m) \mid R'(p_1,\ldots , p_m)\rangle \subset \Gamma$ to denote the $(k-m)$-cube subgroup constructed in Lemma \ref{lem:k-1-cube}. Then
	\[
	\Gamma = \big(\big(\big( \Gamma'(1) \ast_{\langle X'(1) \cap X'(2) \rangle} \Gamma'(2) \big) \ast_{\langle X'(3)\rangle} \Gamma'(3) \big) \ast_{\langle X'(4)\rangle} \cdots \big)\ast_{\langle X'(k) \rangle} \Gamma'(k),
	\]
	where $\ast_G$ is the amalgamated free product over the group $G$.
\end{prop}

\begin{proof}
	Firstly, write $G_2 := \Gamma'(1) \ast_{\langle X'(1) \cap X'(2) \rangle} \Gamma'(2)$, and then 
	\[
	G_{i+1} := G_i \ast_{\langle X'(i)\rangle} \Gamma'(i),
	\]
	for all $2 \leq i \leq k-1$. Then $G_2$ has generating set $X$, and relation set $R'(1) \cup R'(2)$. At each step, we amalgamate over the free group generated by the intersection of $X$ with $X'(i)$, which is $X'(i)$. Hence each $G_i$ is generated by $X$, and has relation set $R'(1) \cup \cdots \cup R'(i)$. But $R'(1) \cup \cdots \cup R'(k) = R$, and so $G_k = \Gamma$.
\end{proof}

\begin{rem}
	It is important to note that the converse is not true---in general it is difficult to find a family of $k$-cube groups whose amalgamated product over the subgroups generated by their pairwise intersections forms a $(k+1)$-cube group.
\end{rem}

\begin{prop}\label{prop:prod_of_trees}
	A group $\Gamma$ is a $k$-cube group if and only if it is a torsion-free $\prod_{i=1}^k \tilde{A}_1$ group, that is, one which acts freely and transitively on the set of vertices of the product of $k$ trees.
\end{prop}

\begin{proof}
	The proof that a torsion-free group of type $\prod_{i=1}^k \tilde{A}_1$ is a $k$-cube group follows the same argument as that of Theorem 3.4 in \cite{KimRob2002}. By the same proof, the fact that a $k$-cube group $\Gamma$ acts on a product of $k$-many trees follows from \ref{prop:clique} and considerations in \cite{BriWis1999}. It is enough to note that elements of $\Gamma$ correspond to paths in the $1$-skeleton of $\Delta$. Then to show that the action is free and transitive, the remainder of Theorem 3.4 of \cite{KimRob2002} can be easily generalised from $k=2$ to arbitrary $k$.
\end{proof}

\begin{rem}
	Vdovina, in \cite{VdoPrep}, used Proposition \ref{prop:prod_of_trees} as the definition of a $k$-cube group.
\end{rem}

\begin{ex}\label{ex:M}
	The $k$-dimensional cube complex $\mathcal{M}$ with adjacency structure $E_1,\ldots , E_k$ constructed in Definition \ref{def:cube_complex} is a $k$-cube group. Indeed, each $k$-cube group yields such a complex with a single vertex, by a relatively obvious process generalised from those in \cite[\S 6.1]{BurMoz2000}, and \cite[\S 4.1]{KhaVdo2018}. We may henceforth regard a $k$-cube group $\Gamma$ both algebraically, and geometrically as the corresponding cube complex with edges labelled by elements of $\Gamma$. If a clear distinction needs to be made, we may write $\mathcal{M}(\Gamma)$ for the geometric realisation of the cube complex.
\end{ex}

\begin{ex}\label{ex:26_cubes}
	Consider the group $\Gamma_{\lbrace 3,5,7 \rbrace}$ from Example 3.17 of \cite{RunStiVdo2019}, defined as follows:
	\[
	\Gamma_{\lbrace 3,5,7 \rbrace} := \langle a_1,a_2, b_1,b_2,b_3, c_1,c_2,c_3,c_4 \mid R \rangle,
	\]
	where
	\begin{align*}
		R := \big\lbrace &a_1 b_1 a_2 b_2, a_1 b_2 a_2 b_1\inv,\, a_1 b_3 a_2\inv b_1,\, a_1 b_3\inv a_1 b_2\inv,\, a_1 b_1\inv a_2\inv b_3,\, a_2 b_3 a_2 b_2\inv, \\ 
		&a_1 c_1 a_2\inv c_2\inv,\, a_1 c_2 a_1 \inv c_3,\, a_1 c_3 a_2\inv c_4\inv,\, a_1 c_4 a_1 c_1\inv, \\
		&a_1 c_4\inv a_2 c_2,\, a_1 c_3\inv a_2 c_1,\, a_2 c_3 a_2 c_2\inv,\, a_2 c_4 a_2\inv c_1, \\
		&c_1 b_1 c_3 b_3\inv,\, c_1 b_2 c_4 b_2\inv ,\, c_1 b_3 c_4\inv b_2,\, c_1 b_3\inv c_4 b_3,\, c_1 b_2\inv c_2 b_1,\, c_1 b_1\inv c_4 b_1\inv, \\
		&c_2 b_2 c_3\inv b_3\inv,\, c_2 b_3 c_4 b_1,\, c_2 b_3\inv c_3 b_3,\, c_2 b_2\inv c_3 b_2,\, c_2 b_1\inv c_3 b_1\inv,\, c_3 b_1 c_4 b_2	 \big\rbrace .
	\end{align*}
	This is a $3$-cube group with adjacency structure $\lbrace a_i,a_i\inv\rbrace, \lbrace b_i,b_i\inv \rbrace, \lbrace c_i,c_i\inv\rbrace$ (Figure \ref{fig:example}).
\end{ex}

\begin{figure}[h]
	\begin{center}
		\includegraphics[scale=0.9]{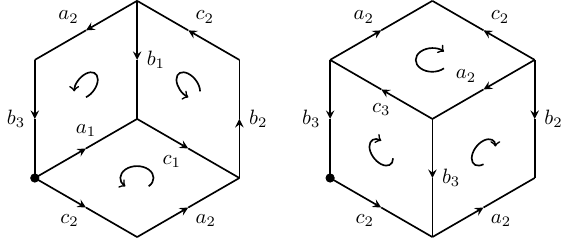}
	\end{center}
	\captionsetup{singlelinecheck=off}
	\caption[.]{Geometric realisation of the pointed, oriented cube in the cube complex corresponding to $\mathcal{S}_3\big(\Gamma_{\lbrace 3,5,7\rbrace}\big)$ labelled 
		\begin{multline*}
			\big[\big[a_1,b_1\inv,a_2,b_3\big], \big[a_2,c_1\inv,a_1\inv,c_2\big], \big[b_2,c_2,b_1,c_1\big], \\ \big[a_2\inv,b_3\inv,a_2\inv,b_2\big],  \big[a_2,c_3,a_2,c_2\inv\big], \big[b_3,c_2\inv,b_3\inv,c_3\inv\big]\big].
		\end{multline*}
		The choice of basepoints and orientations of the cubes and their faces ($2$-cells) is arbitrary, but must remain consistent across the entire complex.
	}\label{fig:example}
\end{figure}

\begin{defn}\label{def:adjacency_functions}
	Let $\Gamma$ be a $k$-cube group with adjacency structure $E_1,\ldots , E_k$, and let $\Delta$ be the rank-$k$ affine building which is the $k$-dimensional cube complex corresponding to $T(|E_1|) \times \cdots \times T(|E_k|)$. We identify elements of $\Gamma$ with vertices of $\Delta$, such that the set $\mathcal{S}_k$ can be identified with the set of pointed, oriented chambers ($k$-cubes) of $\Delta$.
	
	Let $p \in \lbrace 1,\ldots , k\rbrace$ be fixed, and let $A,B \in \mathcal{S}_k$ be the pointed $k$-cubes $\square(u,v_J,u^J,w_J)$,  $\square(x,y_J,x^J,z_J)$, respectively, for $J \subseteq (\lbrace 1,\ldots , k\rbrace \setminus \lbrace p\rbrace)$. We define \textbf{adjacency matrices} $M_1, \ldots , M_k$ to be square matrices with rows and columns indexed by $\mathcal{S}_k$, and with $AB$-th entries given by:
	\begin{itemize}
		\item $M_p(A,B):= 1$ if both of the following criteria are satisfied:
		\begin{enumerate}[label=(\roman*)]
			\item $v_j^L = \big(z_j^L\big)\inv$ and $w_j^L = \big(y_j^L\big)\inv $ for all $j \neq p$ with $j \notin L$,
			\item $u^L \neq (x^L)\inv$,
		\end{enumerate}
		for all $L\subseteq (\lbrace 1,\ldots , k\rbrace \setminus \lbrace p\rbrace)$ with $|L| \geq 0$. We define $M_p(A,B) := 0$ otherwise.
		\item $M_i(A,B) := 1$, for each $i \neq p$, if all of the following criteria are satisfied:
		\begin{enumerate}[label=(\roman*)]
			\item $(x^L)\inv = u^{L\cup \lbrace i \rbrace}$,
			\item $(y_j^L)\inv = w_j^{L \cup \lbrace i \rbrace}$, 
			\item $(z_j^L)\inv = v_j^{L \cup \lbrace i \rbrace}$, 
		\end{enumerate}
		for all $L\subseteq (\lbrace 1,\ldots , k\rbrace \setminus \lbrace i,p\rbrace)$ with $|L| \geq 0$, and all $j$ with $j \notin (L\cup \lbrace i,p\rbrace )$,
		\begin{enumerate}[label=(\roman*)]\setcounter{enumi}{3}
			\item $v_i^L \neq (y_i^L)\inv$ and $w_i^L \neq (z_i^L)\inv$, for all $L\subseteq (\lbrace 1,\ldots , k\rbrace \setminus \lbrace i,p\rbrace)$ with $|L| \geq 0$.
		\end{enumerate}
		We define $M_i(A,B) := 0$ otherwise.
	\end{itemize}
	For each $i \in \lbrace 1,\ldots , k\rbrace$, we say that $B$ is \textbf{adjacent in the} $E_i$ \textbf{direction}, or $E_i$\textbf{-adjacent}, to $A$ whenever $M_i(A,B) = 1$ (Figure \ref{fig:adjacency}).
\end{defn}

\begin{figure}[h]
	\begin{center}
		\includegraphics[scale=0.9]{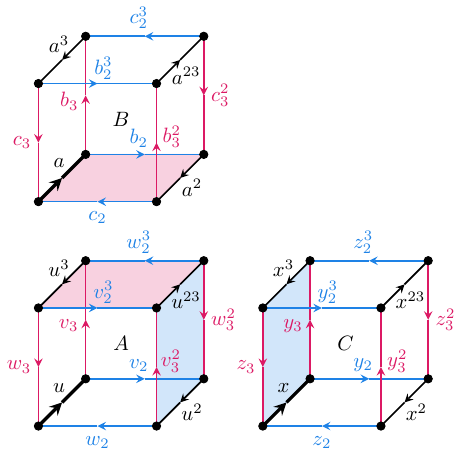}
	\end{center}
	\captionsetup{singlelinecheck=off}
	\caption[.]{Let $a^L,u^L,x^L \in E_1$, $(\ast)_2^L \in E_2$, and $(\ast)_3^L \in E_3$, where $E_1,E_2,E_3$ form the adjacency structure of some $3$-cube group. Consider the pointed $3$-cubes $A,B,C \in \mathcal{S}_3$. If $B$ is $E_3$-adjacent to $A$ (resp. $C$ is $E_2$-adjacent to $A$), then the magenta (resp. blue) $2$-faces above coincide. In general, for a $k$-cube group, $M_i(A,B) = 1$ implies that some $(k-1)$-faces of the geometric realisations of $A$ and $B$ in the corresponding cube complex coincide.
	}\label{fig:adjacency}
\end{figure}

\begin{lem}\label{lem:Mi}
	Let $\Gamma$ be a $k$-cube group with adjacency structure $E_1,\ldots , E_k$. Then each of the adjacency matrices $M_1,\ldots , M_k$ has entries in $\lbrace 0,1\rbrace$, and has at least three non-zero entries in each row.
\end{lem}

\begin{proof}
	Consider the pointed $k$-cube $A := \square(u,v_J,u^J,w_J) \in \mathcal{S}_k(\Gamma)$, as constructed in Definition \ref{def:BM}. Since $|E_i| \geq 4$ for all $i$, and by property \textbf{C1}, we can find some $k$-cube $B := \square\big(x\inv, (z_J)\inv, (x^J)\inv, (v_J)\inv\big)$, where $x\inv \neq u$. Then \textbf{C2} implies that $(x^j)\inv \neq u^j$ for all $j \in J$. It follows that $M_p(A,B) = 1$, and a similar argument can be used to find a $k$-cube $C$ with $M_i(A,C) = 1$, for $i \neq p$.
	
	Hence, in each row in each $M_1,\ldots,M_k$, there are at least three non-zero entries, and by definition these are $\lbrace 0,1 \rbrace$-matrices.
\end{proof}

\begin{figure}[h]
	\begin{center}
		\includegraphics[scale=0.75]{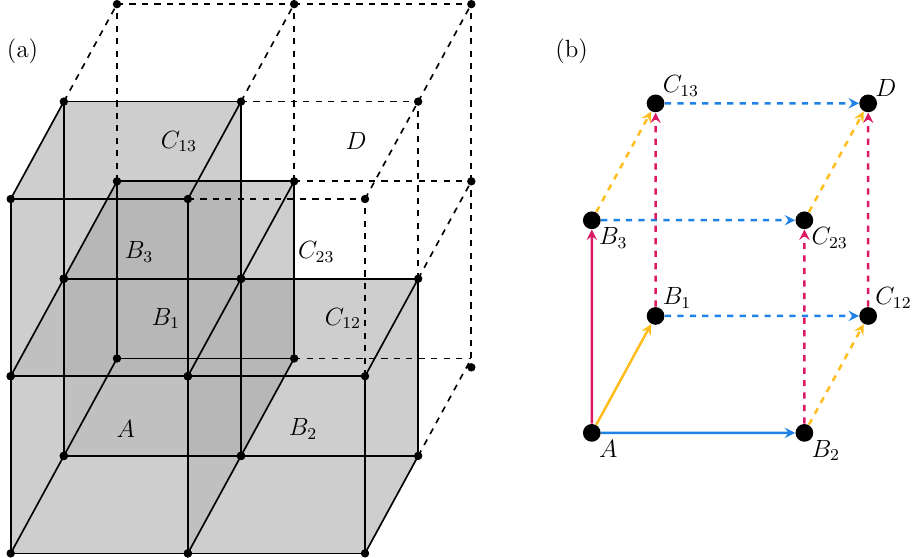}
	\end{center}
	\captionsetup{singlelinecheck=off}
	\caption[.]{\begin{enumerate}[label=(\alph*)]
			\item The Unique Common Extension Property says that, given four adjacent $k$-cubes $A,B_1,B_2,B_3$ arranged as above (grey), we can \textit{uniquely} find four more $k$-cubes $C_{12},C_{13},C_{23},D$ (dashed) to complete a $2 \times 2 \times 2$ arrangement.
			\item In the second figure, we assign each cube from \ref{fig:ufp}(a) a vertex, and draw a directed arrow of colour $i$ between two vertices $X,Y$ whenever $M_i(X,Y) = 1$, that is, if $Y$ is $E_i$-adjacent to $X$. Then, given three arrows originating at $A$, we can find unique arrows (dashed) to complete the commuting cube diagram.
		\end{enumerate}
	}\label{fig:ufp}
\end{figure}

\begin{defn}\label{def:ufp}
	Let $k \geq 3$, and let $\Gamma$ be a $k$-cube group with adjacency structure $E_1,\ldots E_k$ and adjacency matrices $M_1,\ldots , M_k$. Let $A,B_p,B_q,B_r$ be pointed $k$-cubes in $\mathcal{S}_k(\Gamma)$ such that $M_p(A,B_p) = M_q(A,B_q) = M_r(A,B_r) = 1$ for some $p,q,r$. We say that the matrices $M_i$ satisfy the \textbf{Unique Common Extension Property} if we can find \textit{unique} $k$-cubes $C_{pq}, C_{pr}, C_{qr}, D \in \mathcal{S}_k$ such that each of
	\[
	M_p(B_q,C_{pq}),\, M_p(B_r,C_{pr}),\, M_q(B_p,C_{pq}),\, M_q(B_r,C_{pr}),\, M_r(B_p,C_{pr}),\, M_r(B_q,C_{qr}),
	\]
	and each of
	\[
	M_p(C_{qr},D),\, M_q(C_{pr},D),\, M_r(C_{pq},D)
	\]
	is equal to $1$ (Figure \ref{fig:ufp}(a)). In the case where $k=2$, let $A,B,C \in \mathcal{S}_2$ be such that $M_1(A,B) = M_2(A,C) = 1$. Then $M_1,M_2$ satisfy the Unique Common Extension Property if there exists a unique $D \in \mathcal{S}_2$ such that $M_2(B,D) = M_1(C,D) = 1$.
\end{defn}

Referring to the example in Figure \ref{fig:adjacency}, the Unique Common Extension Property would suggest the existence of a unique cube $D$ which is simultaneously $E_2$-adjacent to $B$, and $E_3$-adjacent to $C$.

\begin{rem}\label{rem:ufp}\text{}
	\begin{enumerate}[label=(\roman*)]
		\item One might notice that the definition for the Unique Common Extension Property could be extended to deal with $B_1,\ldots , B_k \in \mathcal{S}_k$, such that each $B_i$ is $E_i$-adjacent to $A$. It turns out by Lemma 1.4 in \cite{RobSte1999}, however, that having unique common extensions given three $k$-cubes $B_p,B_q,B_r$ initially adjacent to $A$ as above, is enough to imply unique common extensions for any number of initial $B_i$.
		
		\item The Unique Common Extension Property is formulated slightly differently to the factorisation property of \textit{rank-$k$ graphs} (c.f. Definition \ref{def:k-graph}, and \cite[1.1]{KumPas2000}). By property \textbf{C2}, any two adjacent sides of a square in the complex $\mathcal{M}(\Gamma)$ uniquely define the square. Then, any three adjacent and mutually perpendicular edges uniquely determine a cube.
		
		The link at each vertex of $\mathcal{M}(\Gamma)$ is a clique complex of a complete $k$-partite graph, so given $k$-cubes $A$ and $D$ arranged as in Figure \ref{fig:ufp}, then the remaining cubes $B_i$, $C_{ij}$ are determined by $A$ and $D$ (Figure \ref{fig:uce}).
	\end{enumerate}
	
	\begin{figure}[h]
		\begin{center}
			\includegraphics[scale=0.75]{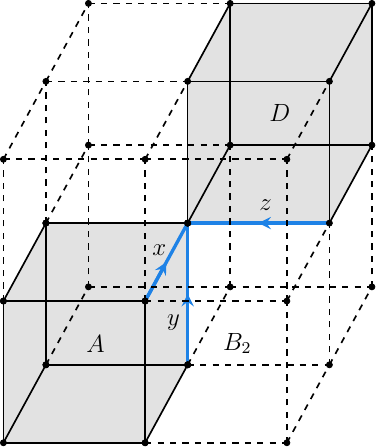}
		\end{center}
		\captionsetup{singlelinecheck=off}
		\caption[.]{Given cubes $A$ and $D$ arranged as above, then $B_2$ is uniquely-determined by edges $x,y,z$. In turn, each of the remaining cubes in the diagram are determined. This is equivalent to the associativity property of rank-$k$ graphs.
		}\label{fig:uce}
	\end{figure}	
	
\end{rem}

\begin{prop}\label{prop:ufp}
	Let $\mathcal{M}$ be the $k$-dimensional cube complex with adjacency structure $E_1,\ldots , E_k$, as constructed in Definition \ref{def:cube_complex}. If we regard $\mathcal{M}$ as a $k$-cube group, then its adjacency matrices $M_1,\ldots , M_k$ commute, and satisfy the Unique Common Extension Property.
\end{prop}

\begin{proof}
	If $k=2$, this is Lemma 4.1 from \cite{KimRob2002}.
	
	Suppose, then, that $k \geq 3$. Fix $p \in \lbrace 1,\ldots , k\rbrace$ and let $A := \square(u,v_J,u^J,w_J) \in \mathcal{S}_k$, for $J \subseteq (\lbrace 1,\ldots , k\rbrace \setminus \lbrace p \rbrace)$. Define three more $k$-cubes $B_1 := \square(a,b_J,a^J,c_J)$, $B_2 := \square(r,s_J,r^J,t_J)$, $B_3 := \square(x,y_J,x^J,z_J) \in \mathcal{S}_k$ and suppose, without loss of generality, that $M_i(A,B_i) = 1$ for each $i \in \lbrace 1,2,3 \rbrace$, and that $p \notin \lbrace 1,2,3 \rbrace$ (these are purely for notational convenience: the proof is identical for any three distinct $i \in \lbrace 1,\ldots , k\rbrace$, or if some $i=p$).
	
	Then by Definition \ref{def:adjacency_functions}, we have $(a^L)\inv = u^{L \cup \lbrace 1 \rbrace}$, $(b_j^L)\inv = w_j^{L \cup \lbrace 1 \rbrace}$, $(c_j^L)\inv = v_j^{L \cup \lbrace 1 \rbrace}$, and so on, for all compatible $L$ and $j \in J$.
	
	By Lemma \ref{lem:Mi}, we can find a $k$-cube $C_{21} \in \mathcal{S}_k$ such that $M_2(B_1,C_{21}) = 1$. Then
	\[
	C_{21} = \square\Big(u^{L \cup \lbrace 1 \rbrace \cup \lbrace 2 \rbrace}, v_J^{L \cup \lbrace 1 \rbrace \cup \lbrace 2 \rbrace}, u^{L \cup \lbrace 1 \rbrace \cup \lbrace 2 \rbrace \cup \lbrace j \rbrace}, w_J^{L \cup \lbrace 1 \rbrace \cup \lbrace 2 \rbrace} \Big).
	\]
	Similarly, we can find some $C_{12} \in \mathcal{S}_k$ such that $M_1(B_1,C_{12}) = 1$. But this can be seen to equal $C_{21}$, and so the matrices $M_1, M_2$ commute. Indeed, we can show in an identical manner that all of the matrices $M_1,\ldots , M_k$ commute.
	
	Finally, consider a $k$-cube $D \in \mathcal{S}_k$ such that $M_3(C_{21},D) = 1$; such a $k$-cube exists which satisfies Definition \ref{def:adjacency_functions}(iv) by Lemma \ref{lem:Mi}. Then
	\[
	D = \square\Big( \big( u^{L \cup \lbrace 1,2,3 \rbrace}\big)\inv, \Big( v_j^{L \cup \lbrace 1,2,3 \rbrace}\Big)\inv, \Big(u^{L \cup \lbrace 1,2,3 \rbrace \cup \lbrace j \rbrace}\Big)\inv, \Big(w_j^{L \cup \lbrace 1,2,3 \rbrace}\Big)\inv \Big),
	\]
	and it is clear that $D$ is also the unique $k$-cube such that $M_1(C_{32},D) = M_2(C_{31},D) = 1$.
\end{proof}

\section{Higher-rank graphs}\label{S:higher-rank}

\begin{defn}\label{def:k-graph}
	Let $\Lambda$ be a category such that $\Ob(\Lambda)$ and $\Hom(\Lambda)$ are countable sets (that is, a \textit{countable small category}), and identify $\Ob(\Lambda)$ with the identity morphisms in $\Hom(\Lambda)$. For a morphism $\lambda \in \Hom_\Lambda(u,v)$, we define range and source maps $r(\lambda) := v$ and $s(\lambda) := u$ respectively.
	
	Let $d : \Lambda \rightarrow \mathbb{N}^k$ be a functor, called the \textbf{degree map}, and let $\lambda \in \Hom(\Lambda)$. We call the pair $(\Lambda, d)$ a \textbf{rank-$k$ graph} (or simply a $k$\textbf{-graph}) if, whenever $d(\lambda) = \mathbf{m} + \mathbf{n}$ for some $\mathbf{m},\mathbf{n} \in \mathbb{N}^k$, we can find \textit{unique} elements $\mu, \nu \in \Hom(\Lambda)$ such that $\lambda = \nu \mu$, and $d(\mu)=\mathbf{m}$, $d(\nu) = \mathbf{n}$. Note that for $\mu$, $\nu$ to be composable, we must have $r(\mu) = s(\nu)$.
	
	For $\mathbf{n} \in \mathbb{N}^k$, we write $\Lambda^\mathbf{n} := d\inv (\mathbf{n})$; by the above property, we have that $\Lambda^\mathbf{0} = \Ob(\Lambda)$, and we call the elements of $\Lambda^\mathbf{0}$ the \textbf{vertices} of $(\Lambda, d)$ \cite{KumPas2000}.
	
	Let $(\Lambda,d)$ be a rank-$k$ graph, let $\mathbf{n} \in \mathbb{N}^k$, and let $v \in \Lambda^\mathbf{0}$. Write $v\Lambda^\mathbf{n}$ for the set of morphisms in $\Lambda^\mathbf{n}$ which map onto the vertex $v$, that is, $v\Lambda^\mathbf{n} := \lbrace \lambda \in \Lambda^\mathbf{n} \mid r(\lambda) = v\rbrace$. We say that $(\Lambda,d)$ is \textbf{row-finite} if each set $v\Lambda^\mathbf{n}$ is finite, and that $(\Lambda,d)$ has \textbf{no sources} if each $v\Lambda^\mathbf{n}$ is non-empty.
\end{defn}

\begin{rem}\label{rem:directed_graphs}
	If $E$ is a directed graph on $n$ vertices, we can construct an $n \times n$ incidence matrix $M_E(i,j)$ with $ij$-th entry $1$ if there is an edge from $i$ to $j$, and $0$ otherwise.
	
	If $E_1,\ldots , E_k$ are directed graphs with the same vertex set, and such that their associated incidence matrices $M_1,\ldots ,M_k$ commute and satisfy the Unique Common Extension Property, then we can construct a rank-$k$ graph out of the graphs $E_i$, as in \cite{HazRaeSimWeb2013}.
\end{rem}

\begin{thm}\label{thm:k-graph}
	Let $\Gamma$ be a $k$-cube group with adjacency structure $E_1,\ldots , E_k$. Then $\Gamma$ induces a row-finite rank-$k$ graph $\mathcal{G}(\Gamma)$ with no sources.
\end{thm}

\begin{rem}
	To the reader who has not come across higher-rank graphs in the past, it may seem counter-intuitive that a rank-$k$ graph be a countably-infinite category, while a $k$-cube group $\Gamma$ comprises finite data.
	
	The adjacency structure on $\Gamma$ induces a family $\mathcal{E}(\Gamma)$ of distinctly-coloured directed graphs $\mathcal{E}_1,\ldots , \mathcal{E}_k$ on the same vertices. The shared set of vertices is the set of $k$-cubes of $\Gamma$, and the incidence matrices are given by $M_1, \ldots , M_k$, respectively. It is imperative to stress that this is not the same as the cube complex $\mathcal{M}(\Gamma)$ from Example \ref{ex:M}. In \cite{EvaSim2012}, the collection of graphs $\mathcal{E}(\Gamma)$ is called a $1$-\textit{skeleton}.
	
	To view $\mathcal{E}(\Gamma)$ as a rank-$k$ graph, we must consider the \textit{$k$-dimensional paths $\mathcal{G}(\Gamma)$} of $\mathcal{E}(\Gamma)$. If $M_i M_j(v,w) = M_j M_i (v,w) = 1$ for some vertices $v,w \in \mathcal{E}(\Gamma)$, then there is a unique $ij$-coloured path, and a unique $ji$-coloured path, from $v$ to $w$. We then identify these paths and the quotient $\mathcal{G}(\Gamma)$ of the path category $\mathcal{E}(\Gamma)^*$ by this relation is then a $k$-graph, as long as the Unique Common Extension Property is also satisfied (c.f. \cite{HazRaeSimWeb2013}).
	
	The degree $d(\lambda)$ (called the \textit{shape} in \cite{RobSte1999}) of a $k$-dimensional path $\lambda$ is a tuple whose $i$-th entry is the total length of $\lambda$ restricted to colour $i$, that is, to edges labelled by elements of $\mathcal{E}_i$. Together, the pair $(\mathcal{G}(\Gamma),d)$ is a rank-$k$ graph, since it satisfies the factorisation property of Definition 1.1 in \cite{KumPas2000} (consult Figure \ref{fig:uce} for an illustration). We prove Theorem \ref{thm:k-graph} like so:
\end{rem}

\begin{proof}[Proof of Theorem \ref{thm:k-graph}]
	Let $\mathcal{G}(\Gamma)$ be a collection of directed graphs $\mathcal{E}_1,\ldots , \mathcal{E}_k$, each on the set of vertices $\mathcal{S}_k(\Gamma)$, and with incidence matrices $M_1,\ldots , M_k$ respectively. From Remark 2.3 in \cite{FowSim2001} and \cite[\S 1]{RobSte1999}, it is sufficient that, for each $i,j,l \in \lbrace 1,\ldots , k\rbrace$ with $i < j < l$:
	\begin{enumerate}[label=(\roman*)]
		\item $M_i$ is non-zero,
		\item $M_i M_j = M_j M_i$, 
		\item Each of $M_i$, $M_i M_j$, and $M_i M_j M_l$ has entries in $\lbrace 0,1\rbrace$.
	\end{enumerate}
	But each $M_i$ is non-zero by Lemma \ref{lem:Mi}, and the matrices are finite-dimensional, commute, and satisfy the Unique Common Extension Property by Proposition \ref{prop:ufp}. Hence $\mathcal{G}(\Gamma)$ can be regarded as a rank-$k$ graph.
\end{proof}

\begin{rem}	
	Robertson and Steger in \cite{RobSte1999}, \cite{RobSte2001} initially considered those rank-$2$ graphs whose incidence matrices have entries in $\lbrace 0,1 \rbrace$ (though they did not name them as such). Then in \cite[Theorem 3.5, Corollary 3.10]{APS2006}, using the dual graph construction, the $\cst$-algebra of any $k$-graph $\Lambda$ is isomorphic to the $\cst$-algebra of a $k$-graph $p \cdot \Lambda$ whose co-ordinate matrices all have entries in $\lbrace 0,1 \rbrace$, and where $p = (1, \ldots , 1) \in \mathbb{N}^k$. Our $k$-cube groups induce the another example of rank-$k$ graphs with $\lbrace 0,1 \rbrace$-incidence-matrices for arbitrary $k \geq 2$.
\end{rem}

We associate a $\cst$-algebra to a rank-$k$ graph as follows:

\begin{defn}\label{def:graph_algebra}
	Let $\Lambda = (\Lambda,d)$ be a row-finite rank-$k$ graph with no sources. We define the \textbf{rank-$k$ graph $\cst$-algebra} $\mathcal{A}(\Lambda)$ to be the universal $\cst$-algebra generated by a family $\lbrace s_\lambda \mid \lambda \in \Lambda\rbrace$ of \textit{partial isometries} (that is, operators $s_\lambda$ whose restriction to $(\ker s_\lambda)^\perp$ are isometries) which have the following properties:
	\begin{enumerate}[label=(\roman*)]
		\item The set $\big\lbrace s_v \mid v \in \Lambda^\mathbf{0}\big\rbrace$ satisfies $(s_v)^2 = s_v = s_v^*$ and $s_u s_v = 0$ for all $u \neq v$.
		\item If $r(\lambda) = s(\mu)$ for some $\lambda, \mu \in \Lambda$, then $s_{\mu\lambda} = s_\mu s_\lambda$.
		\item For all $\lambda \in \Lambda$, we have $s_\lambda^* s_\lambda = s_{s(\lambda)}$.
		\item For all vertices $v \in \Lambda^\mathbf{0}$ and $\mathbf{n} \in \mathbb{N}^k$, we have:
		\[
		s_v = \sum_{v\lambda \in \Lambda^\mathbf{n}} s_\lambda s_\lambda^*.
		\]
	\end{enumerate}
	Note that without the row-finiteness condition, property (iv) is not well-defined.
\end{defn}

\section{Spectral sequences and K-theory}\label{S:spectral}

We make extensive use of Theorem 3.15 from \cite{Eva2008}, displayed here as Theorem \ref{thm:gwion}; in the examples presented in this paper, we principally consider the special cases where $k=3$ or $k=4$. For $k=3$, we make use of the relevant work of \cite{Eva2008}, and we derive analogous results for $k=4$ and $k=5$ in Propositions \ref{prop:4-rank_ss}, $\ref{prop:5-rank_ss}$ and Corollary \ref{cor:4-rank}.

The proofs make use of so-called \textit{spectral sequences}, generalisations of chain complexes; we direct the unfamiliar reader to \cite{Mcc2000} for more detailed background information, but offer an overview here.

\begin{defn}
	Let $\mathcal{C}$ be an Abelian category. A \textbf{spectral sequence (of homological type)} consists of a family $\lbrace (E^r, d^r)\rbrace$ of bigraded objects 
	\[
	E^r := \bigoplus_{p,q \in \mathbb{Z}} E_{p,q}^r
	\]
	in $\mathcal{C}$, and maps
	\[
	d^r : E_{p,q}^r \longrightarrow E_{p-r, q+r-1}^r, \quad\text{and}\quad d^r : E_{p+r, q-r+1}^r \longrightarrow E_{p,q}^r,
	\]
	called \textbf{differentials}, which are of bidegree $(-r,r-1)$, and which satisfy $d^r \circ d^r = 0$. We insist that
	\[
	E_{p,q}^{r+1} \cong H(E_{p,q}^r) := \frac{\ker\big(d^r : E_{p,q}^r \longrightarrow E_{p-r,q+r-1}^r \big)}{\im\big(d^r : E_{p+r,q-r+1}^r \longrightarrow E_{p,q}^r \big)}.
	\]
	The collections $(E_{p,q}^r)$ for fixed $r$ are known as the \textbf{sheets} of the spectral sequence. We move to the next sheet by taking the homology $H$, defined above. We call a spectral sequence \textbf{bounded} if the sequence of objects $E_{p,q}^r$  stabilises as $r \rightarrow \infty$; we denote this limit by $E_{p,q}^\infty$, and call it the \textbf{stable value}.
	
	We say that a bounded spectral sequence \textbf{converges} to a family of $\mathbb{Z}$-modules $\lbrace \mathcal{K}_* \rbrace$ if there exists a finite ascending filtration of modules
	\begin{equation} \label{eq:filtration}
		0 = F_s(\mathcal{K}_*) \subseteq \cdots \subseteq F_{p-1}(\mathcal{K}_*) \subseteq F_{p}(\mathcal{K}_*) \subseteq F_{p+1}(\mathcal{K}_*) \subseteq \cdots \subseteq F_t(\mathcal{K}_*) = \mathcal{K}_*,
	\end{equation}
	and an isomorphism
	\begin{equation} \label{eq:iso}
		E_{p,q}^\infty \cong F_p(\mathcal{K}_{p+q}) / F_{p-1}(\mathcal{K}_{p+q}),
	\end{equation}
	for every pair $(p,q)$.
\end{defn}

Given a general chain complex $A := \cdots \rightarrow A_{i+1} \overset{\del_{i+1}}{\longrightarrow} A_i \overset{\del_i}{\longrightarrow} A_{i-1} \rightarrow \cdots$, we frequently write $H_i(A)$ to denote the $i$-th homology $\ker(\del_i)/\im(\del_{i+1})$.

\begin{thm}[Evans 2008]\label{thm:gwion}
	Define the sets
	\[
	N_l := \begin{cases}
		\big\lbrace \bmu := (\mu_1 , \ldots , \mu_l) \in \lbrace 1, \ldots , k\rbrace^l \mid \mu_1 < \cdots < \mu_l \big\rbrace &\text{if } l \in \lbrace 1,\ldots , k\rbrace, \\
		\lbrace * \rbrace & \text{if } l=0, \\
		\emptyset & \text{otherwise},
	\end{cases}
	\]	
	and for $l \in \lbrace 1, \ldots , k \rbrace$ and $\bmu \in N_l$, define
	\[
	\bmu^i := \begin{cases}
		(\mu_1, \ldots , \hat{\mu}_i , \ldots \mu_l ) \in N_{l-1} & \text{if } l > 1, \\
		* & \text{if } l = 1.
	\end{cases}
	\]
	Let $\Lambda$ be a row-finite $k$-graph with no sources. Then there exists a spectral sequence $\lbrace (E^r, d^r)\rbrace$ converging to $\mathcal{K}_* (\mathcal{A}(\Lambda))$ with $E_{p,q}^\infty \cong E_{p,q}^{k+1}$, and
	\[
	E_{p,q}^2 \cong \begin{cases}
		H_p(\mathcal{D}_k) & \text{if } p \in \lbrace 0,1,\ldots , k\rbrace \text{ and } q \text{ is even,} \\
		0 & \text{otherwise,}
	\end{cases}
	\]
	where $\mathcal{D}_k$ is the chain complex with
	\[
	(\mathcal{D}_k)_p := \begin{cases}
		\bigoplus_{\bmu \in N_p} \mathbb{Z}\Lambda^\mathbf{0} & \text{if } p \in \lbrace 0,1,\ldots , k\rbrace , \\
		0 & \text{otherwise,}
	\end{cases}
	\]
	and whose differentials $\del_p : (\mathcal{D}_k)_p \longrightarrow (\mathcal{D}_k)_{p-1}$ are defined as
	\[
	\bigoplus_{\bmu \in N_p} m_{\bmu} \longmapsto \bigoplus_{\lambda\in N_{p-1}} \sum_{\bmu \in N_p} \sum_{i=1}^p (-1)^{i+1} \delta_{\lambda,\bmu^i} (I-M_{\mu_i}^T)m_{\bmu},
	\]
	for $p \in \lbrace 1,\ldots , k\rbrace$.\qed
\end{thm}

\begin{prop}[$k=3$, Evans 2008]\label{prop:3-rank_ss}
	Let $\Lambda$ be a row-finite $3$-graph with no sources, and let $\mathbb{Z}\Lambda^\mathbf{0}$ be the group of all maps $\Lambda^\mathbf{0} \rightarrow \mathbb{Z}$ with finite support under pointwise addition. Consider the chain complex $\mathcal{D}_3$ defined as follows:
	\[
	0 \longrightarrow \mathbb{Z}\Lambda^\mathbf{0} \overset{\del_3}{\longrightarrow} \bigoplus_{i=1}^3 \mathbb{Z}\Lambda^\mathbf{0} \overset{\del_2}{\longrightarrow} \bigoplus_{i=1}^3 \mathbb{Z}\Lambda^\mathbf{0} \overset{\del_1}{\longrightarrow} \mathbb{Z}\Lambda^\mathbf{0} \longrightarrow 0,
	\]
	where $\del_1$, $\del_2$, $\del_3$ are defined by the block matrices
	\begin{align*}
		\del_1 & := \begin{bmatrix}
			I-M_1^T & I-M_2^T & I-M_3^T
		\end{bmatrix}, \\
		\del_2 & := \begin{bmatrix}
			M_2^T-I & M_3^T-I & 0 \\
			I-M_1^T & 0 & M_3^T-I \\
			0 & I-M_1^T & I-M_2^T
		\end{bmatrix}, \\
		\del_3 & := \begin{bmatrix}
			I-M_3^T \\
			M_2^T-I \\
			I-M_1^T
		\end{bmatrix}.
	\end{align*}
	Then for some subgroups $G_0 \subseteq \coker(\del_1)$ and $G_1 \subseteq \ker(\del_3)$, there exists a short exact sequence
	\[
	0 \longrightarrow \coker(\del_1) / G_0 \longrightarrow K_0(\mathcal{A}(\Lambda)) \longrightarrow \ker(\del_2) / \im(\del_3) \longrightarrow 0,
	\]
	and an isomorphism
	\[
	K_1(\mathcal{A}(\Lambda)) \cong \ker(\del_1) / \im(\del_2) \oplus G_1.
	\]
	\qed
\end{prop}

\begin{corl}[$k=3$, Evans 2008]\label{cor:3-rank}
	In addition to the hypotheses of Proposition \ref{prop:3-rank_ss}:
	\begin{enumerate}[label=(\roman*)]
		\item If $\del_1$ is surjective, then:
		\begin{enumerate}
			\item $K_0(\mathcal{A}(\Lambda)) \cong \ker(\del_2) / \im(\del_3)$,
			\item $K_1(\mathcal{A}(\Lambda)) \cong (\ker(\del_1) / \im(\del_2)) \oplus \ker(\del_3)$.
		\end{enumerate}
		\item If $\bigcap_i \ker\big(I-M_i^T\big) = 0$, then there exists a short exact sequence
		\[
		0 \longrightarrow \coker(\del_1) \longrightarrow K_0(\mathcal{A}(\Lambda)) \longrightarrow \ker(\del_2) / \im(\del_3) \longrightarrow 0,
		\]
		and an isomorphism
		\[
		K_1(\mathcal{A}(\Lambda)) \cong \ker(\del_1) / \im(\del_2).
		\]
		\qed		
	\end{enumerate}
\end{corl}

\begin{prop}[$k=4$]\label{prop:4-rank_ss}
	Let $\Lambda$ be a row-finite $4$-graph with no sources, and consider the chain complex $\mathcal{D}_4$:
	\[
	0 \longrightarrow \mathbb{Z}\Lambda^\mathbf{0} \overset{\del_4}{\longrightarrow} \bigoplus_{i=1}^4 \mathbb{Z}\Lambda^\mathbf{0} \overset{\del_3}{\longrightarrow} \bigoplus_{i=1}^6 \mathbb{Z}\Lambda^\mathbf{0}  \overset{\del_2}{\longrightarrow} \bigoplus_{i=1}^4 \mathbb{Z}\Lambda^\mathbf{0} \overset{\del_1}{\longrightarrow} \mathbb{Z}\Lambda^\mathbf{0} \longrightarrow 0,
	\]
	where $\del_1, \ldots , \del_4$ are the group homomorphisms defined by the block matrices from Theorem \ref{thm:gwion}, namely:
	\begin{align*}
		\del_1 & := \begin{bmatrix}
			I-M_1^T & I-M_2^T & I-M_3^T & I-M_4^T
		\end{bmatrix}, \\
		\del_2 & := \begin{bmatrix}
			M_2^T-I & M_3^T-I & M_4^T-I & 0 & 0 & 0 \\
			I-M_1^T & 0 & 0 & M_3^T-I & M_4^T-I & 0 \\
			0 & I-M_1^T & 0 & I-M_2^T & 0 & M_4^T-I \\
			0 & 0 & I-M_1^T & 0 & I-M_2^T & I-M_3^T
		\end{bmatrix}, \\
		\del_3 & := \begin{bmatrix}
			I-M_3^T & I-M_4^T & 0 & 0 \\
			M_2^T-I & 0 & I-M_4^T & 0 \\
			0 & M_2^T-I & M_3^T-I & 0 \\
			I-M_1^T & 0 & 0 & I-M_4^T \\
			0 & I-M_1^T & 0 & M_3^T-I \\
			0 & 0 & I-M_1^T & I-M_2^T
		\end{bmatrix}, \\
		\del_4 & := \begin{bmatrix}
			M_4^T-I \\ I-M_3^T \\ M_2^T-I \\ I-M_1^T
		\end{bmatrix}.
	\end{align*}
	
	Write $H_i(\mathcal{D}_4) := \ker(\del_i) / \im(\del_{i+1})$, and let $F_2$ be a factor in the ascending filtration of the $\cst$-algebra $\mathcal{A}(\Lambda)$. Then, for some subgroups
	\[
	G_0 \subseteq \coker(\del_1),\quad G_1 \subseteq \ker(\del_4),\quad G_2 \subseteq H_1(\mathcal{D}_4),\quad G_3 \subseteq H_3(\mathcal{D}_4),
	\]
	there exist short exact sequences as follows:
	\begin{enumerate}[label=(\roman*)]
		\item $0 \longrightarrow \coker(\del_1) / G_0 \longrightarrow K_0(\mathcal{A}(\Lambda)) \longrightarrow \dfrac{K_0(\mathcal{A}(\Lambda))}{\coker(\del_1)/G_0} \longrightarrow 0$,
		
		\item $0 \longrightarrow \coker(\del_1)/G_0 \longrightarrow F_2 \longrightarrow \dfrac{\ker(\del_2)}{\im(\del_3)} \longrightarrow 0$,
		
		\item $0 \longrightarrow F_2 \longrightarrow K_0(\mathcal{A}(\Lambda)) \longrightarrow G_1 \longrightarrow 0$,
		
		\item $0 \longrightarrow \dfrac{\ker(\del_1)/\im(\del_2)}{G_2} \longrightarrow K_1(\mathcal{A}(\Lambda)) \longrightarrow G_3 \longrightarrow 0$,
	\end{enumerate}
	and sequence (iii) splits, such that $K_0(\mathcal{A}(\Lambda)) \cong F_2 \oplus G_1$.
\end{prop}

\begin{proof}
	Write $\lbrace (E^r, d^r)\rbrace$ to denote the Kasparov spectral sequence of homological type introduced in \cite{Eva2008}. We know that $\lbrace (E^r, d^r)\rbrace$ is bounded, and that the stable value of $E_{p,q}^r$ is $E_{p,q}^\infty \cong E_{p,q}^5$. The spectral sequence converges to $\mathcal{K}_*(\mathcal{A}(\Lambda))$, so we have the finite ascending filtration (\ref{eq:filtration}) and isomorphism (\ref{eq:iso}) with $E_{p,q}^\infty \cong E_{p,q}^5 = 0$ whenever $p \in (\mathbb{Z}\setminus \lbrace 0,\ldots , 4\rbrace)$ or $q$ is odd.
	
	\textbf{Case I:} Firstly, we turn our attention to $K_0(\mathcal{A}(\Lambda))$. Write $K_0 = K_0(\mathcal{A}(\Lambda)) = \mathcal{K}_*(\mathcal{A}(\Lambda))_{p+q}$, as in Lemma 3.3 of \cite{Eva2008}, and fix the total degree, $p+q$, to be zero.
	
	We have that $E_{p,q}^5 = 0$ unless $p \in \lbrace 0,2,4\rbrace$, since if $p$ is odd and $p+q=0$, then $q$ is odd. Suppose, then, that $p \notin \lbrace 0,2,4\rbrace$, such that $0 = E_{p,q}^5 = F_p(K_0) / F_{p-1}(K_0)$, and hence $F_p(K_0) \cong F_{p-1}(K_0)$. We can deduce that, in our filtration, we have $F_1(K_0) = F_0(K_0)$, and $F_{i+1}(K_0) = F_i(K_0)$ for all $i \geq 2$.
	
	By the same argument, it follows that $F_i(K_0) = 0$ for all $i < 0$, and so the filtration becomes
	\[
	0 \subseteq F_0(K_0) \subseteq F_2(K_0) \subseteq K_0.
	\]
	Next, we consider the non-zero $E_{p,q}^5$ terms. From (\ref{eq:iso}), we have:
	\begin{itemize}
		\item $E_{0,0}^5 \cong F_0(K_0)$,
		\item $E_{2,-2}^5 \cong F_2(K_0) / F_1(K_0) \cong F_2(K_0) / F_0(K_0)$,
		\item $E_{4,-4}^5 \cong F_4(K_0) / F_3(K_0) \cong K_0 / F_2(K_0)$.
	\end{itemize}
	It then follows that we have short exact sequences:
	\begin{enumerate}[label=(\roman*')]
		\item $0 \longrightarrow E_{0,0}^5 \longrightarrow K_0 \longrightarrow K_0 / E_{0,0}^5 \longrightarrow 0$,
		\item $0 \longrightarrow E_{0,0}^5 \longrightarrow F_2(K_0) \longrightarrow E_{2,-2}^5 \longrightarrow 0$,
		\item $0 \longrightarrow F_2(K_0) \longrightarrow K_0 \longrightarrow E_{4,-4}^5 \longrightarrow 0$.
	\end{enumerate}
	
	\textbf{Case II:} $p+q = 1$.
	
	We consider $K_1(\mathcal{A}(\Lambda))$. Note that in order for $E_{p,q}^5$ to be non-zero, we must have $p \in \lbrace 0,\ldots , 4\rbrace$ and $q$ even. But, the only pairs $(p,q)$ of total degree $1$ are $(1,0)$ and $(3,-2)$. Thus, it follows analogously from Proposition 3.17 in \cite{Eva2008} that there is a short exact sequence
	\[
	0 \longrightarrow E_{1,0}^5 \longrightarrow K_1(\mathcal{A}(\Lambda)) \longrightarrow E_{3,-2}^5 \longrightarrow 0.
	\]
	
	\begin{center}
		* \qquad * \qquad *
	\end{center}
	The final step of the proof is to compute the following:
	\begin{enumerate}[label=(\alph*)]
		\item For $K_1(\mathcal{A}(\Lambda))$, the terms $E_{1,0}^5$ and $E_{3,-2}^5$,
		\item For $K_0(\mathcal{A}(\Lambda))$, the terms $E_{0,0}^5$, $E_{2,-2}^5$, and $E_{4,-4}^5$.
	\end{enumerate}
	
	\textbf{Step (a):} $E_{1,0}^5$ and $E_{3,-2}^5$.
	
	We proceed by considering the differentials $d^4$, $d^3$, and $d^2$. Since $E_{p,q}^4 = 0$ whenever $p \in (\mathbb{Z} \setminus\lbrace 1,\ldots , 4\rbrace)$, we necessarily have:
	\[
	d^4 : E_{p,q}^4 \longrightarrow E_{p-4,q+3}^4, \qquad d^4 : E_{p+4,q-3}^4 \longrightarrow E_{p,q}^4,
	\]
	for $p \in \lbrace 0,4\rbrace$. However, in either case we must have $q,q+3$ or $q,q-3$ both even: a contradiction. Hence $d^4$ is the zero map.
	
	Similarly, it follows that the only non-zero components of the $d^3$ differential are
	\[
	d^3 : E_{3,q}^3 \longrightarrow E_{0,q+2}, \qquad d^3 : E_{4,q}^3 \longrightarrow E_{1,q+2}^3,
	\]
	for $q$ even. Furthermore, we can deduce that $d^2$ must also be the zero map, as in Proposition 3.16 of \cite{Eva2008}. Thus, we have:
	\begin{align*}
		E_{1,0}^5 &\cong H\big(E_{1,0}^4\big) = \frac{\ker\big(d^4 : E_{1,0}^4 \rightarrow E_{-3,3}^4\big)}{\im\big(d^4 : E_{5,-3}^4 \rightarrow E_{1,0}^4\big)} = E_{1,0}^4,\\
		E_{1,0}^4 &\cong H\big(E_{1,0}^3\big) = \frac{\ker\big(d^3 : E_{1,0}^3 \rightarrow E_{-2,2}^3\big)}{\im\big( d^3 : E_{4,-2}^3 \rightarrow E_{1,0}^3\big)} = E_{1,0}^3 / \im\big(d^3 : E_{4,-2}^3 \rightarrow E_{1,0}^3\big).
	\end{align*}
	Now, let $G_2$ be a subgroup of $E_{1,0}^3 = H_1(\mathcal{D}_4)$, namely $G_2 := \im\big(d^3 : E_{4,-2}^3 \rightarrow E_{1,0}^3\big)$. Then we have
	\[
	E_{1,0}^3 \cong H\big(E_{1,0}^2\big) = E_{1,0}^2 = H_1(\mathcal{D}_4),
	\]
	and so $E_{1,0}^5 \cong (\ker(\del_1)/ \im(\del_2)) / G_2$.
	
	It remains to compute $E_{3,-2}^5$. We have:
	\begin{align*}
		E_{3,-2}^5 &\cong H\big(E_{3,-2}^4\big) = \frac{\ker\big( d^4 : E_{3,-2}^4 \rightarrow E_{-1,1}^4\big)}{\im\big( d^4 : E_{7,-5}^4 \rightarrow E_{3,-2}^4\big)} = E_{3,-2}^4, \\
		E_{3,-2}^4 &\cong H\big(E_{3,-2}^3\big) = \frac{\ker\big(d^3 : E_{3,-2}^3 \rightarrow E_{0,0}^3\big)}{\im\big(d^3 : E_{6,-4}^3 \rightarrow E_{3,-2}^3\big)} = \ker\big(d_{3,-2}^3\big) \subseteq E_{3,-2}^3.
	\end{align*}
	Now, $E_{3,-2}^3 \cong H(E_{3,-2}^3) \subseteq E_{3,-2}^3 = H_3(\mathcal{D}_4)$, and hence
	\[
	E_{3,-2}^5 \cong \ker\big( d_{3,-2}^3\big) \subseteq E_{3,-2}^3 = H_3(\mathcal{D}_4).
	\]
	Write $G_3 := \ker(d_{3,-2}^3)$, which is a subgroup of $H_3(\mathcal{D}_4)$. Then we have the short exact sequence (iv):
	\[
	0 \longrightarrow H_1(\mathcal{D}_4) / G_2 \longrightarrow K_1(\mathcal{A}(\Lambda)) \longrightarrow G_3 \longrightarrow 0.
	\]
	
	\textbf{Step (b):} $E_{0,0}^5$, $E_{2,-2}^5$, and $E_{4,-4}^5$.
	
	Firstly, consider $E_{0,0}^5$. We know that $E_{0,0}^5 \cong H(E_{0,0}^4) = E_{0,0}^4$, since the differential $d^4$ is the zero map. We also have
	\[
	E_{0,0}^4 \cong H\big(E_{0,0}^3\big) = \frac{\ker\big( d^3 : E_{0,0}^3 \rightarrow E_{-3,0}^3\big)}{\im\big(d^3 : E_{3,-2}^3 \rightarrow E_{0,0}^3\big)} = E_{0,0}^3 / \im\big(d_{3,-2}^3\big).
	\]
	Note that $E_{0,0}^3 \cong H(E_{0,0}^2) = E_{0,0}^2 = H_0(\mathcal{D}_4) = \coker(\del_1)$, so that if we write $G_0 := \im(d_{3,-2}^3)$, we obtain $E_{0,0}^5 = \coker(\del_1) / G_0$. This, together with the sequence (i') above, gives us the sequence (i).
	
	Now, we turn our attention to $E_{2,-2}^5$ and $E_{4,-4}^5$. We know that $E_{2,-2}^5 \cong H(E_{2,-2}^4) = E_{2,-2}^4$, by virtue of $d^4$ being the zero map. We also have
	\[
	E_{2,-2}^4 \cong H\big(E_{2,-2}^3\big) = \frac{\ker\big( d^3 : E_{2,-2}^3 \rightarrow E_{-1,0}^3\big)}{\im\big( d^3 : E_{5,-4}^3 \rightarrow E_{2,-2}^3 \big)} = E_{2,-2}^3 \cong H\big(E_{2,-2}^2\big) = H_2(\mathcal{D}_4),
	\]
	and so $E_{2,-2}^5 \cong H_2(\mathcal{D}_4)$. Together with (ii') and the above, this gives us (ii). We also know that $E_{4,-4}^5 \cong H(E_{4,-4}^4) = E_{4,-4}^5$, and
	\[
	E_{4,-4}^4 \cong H\big(E_{4,-4}^3\big) = \frac{\ker\big( d^3 : E_{4,-4}^3 \rightarrow E_{1,-2}^3\big)}{\im\big( d^3 : E_{7,-6}^3 \rightarrow E_{4,-4}^3 \big)} = \ker\big(d_{4,-4}^3\big) \subseteq E_{4,-4}^3,
	\]
	and so $E_{4,-4}^3 \cong H(E_{4,-4}^2) = E_{4,-4}^2 = H_4(\mathcal{D}_4)$. Writing $G_1 := \ker(d_{4,-4}^3)$, and putting this together with (iii'), we obtain the sequence (iii). We know that $H_4(\mathcal{D}_4)$ is a free Abelian group, and since subgroups of such groups are also free Abelian, it follows that $G_1$ is free Abelian, and sequence (iii) splits.
\end{proof}

\begin{corl}[$k=4$]\label{cor:4-rank}
	In addition to the hypotheses of Proposition \ref{prop:4-rank_ss}:
	\begin{enumerate}[label=(\roman*)]
		\item If $\del_1$ is surjective, then there exists an isomorphism $F_2 \cong \ker(\del_2)/\im(\del_3)$, and the short exact sequences reduce to:
		\begin{enumerate}[label=(\alph*)]
			\item $K_0(\mathcal{A}(\Lambda)) \cong \dfrac{\ker(\del_2)}{\im(\del_3)} \oplus G_1$,
			\item $0 \longrightarrow \dfrac{\ker(\del_1)/\im(\del_2)}{G_2} \longrightarrow K_1(\mathcal{A}(\Lambda)) \longrightarrow \ker(\del_3) / \im(\del_4) \longrightarrow 0$.
		\end{enumerate}
		\item If $\bigcap_i \ker \big( I - M_i^T\big) = 0$, then $K_0(\mathcal{A}(\Lambda)) \cong F_2$, and the sequences reduce to:
		\begin{enumerate}[label=(\alph*)]
			\item $0 \longrightarrow \coker(\del_1) / G_0 \longrightarrow K_0(\mathcal{A}(\Lambda)) \longrightarrow \ker(\del_3) / \im(\del_2) \longrightarrow 0$,
			\item $0 \longrightarrow\ker(\del_1) / \im(\del_2) \longrightarrow K_1(\mathcal{A}(\Lambda)) \longrightarrow G_3 \longrightarrow 0$.
		\end{enumerate}
	\end{enumerate}
\end{corl}

\begin{proof}
	To show (i), suppose that $\del_1$ is surjective, such that $\coker(\del_1) = 0$, and $F_2(K_0) \cong \ker(\del_2)/\im(\del_3)$. Then the split exact sequence (iii) from Proposition \ref{prop:4-rank_ss} gives us (i)(a).
	
	Now, we have $0 = \coker(\del_1) = H_0(\mathcal{D}_4) = E_{0,0}^3$, and so $d^3 : E_{3,-2}^3 \rightarrow E_{0,0}^3$ is the zero map. Hence $\ker(d_{3,-2}^3) = E_{3,-2}^3 = H_3(\mathcal{D}_4)$, and we obtain (i)(b) from Proposition \ref{prop:4-rank_ss}(iv).
	
	To show (ii), suppose that $\bigcap_i \big( I - M_i^T\big) = 0$. Then $\ker(\del_4) = 0$, and hence $G_1 = 0$ and $K_0 \cong F_2$. This gives us (ii)(a).
	
	Now, from Proposition \ref{prop:4-rank_ss}(iv), we have the sequence
	\[
	0 \longrightarrow \frac{\ker(\del_1) / \im(\del_2)}{G_2} \longrightarrow K_1 \longrightarrow G_3 \longrightarrow 0,
	\]
	where
	\begin{align*}
		G_2 & := \im\big( d_{4,-2}^3 : E_{4,-2}^3 \longrightarrow E_{1,0}^3 = H_1(\mathcal{D}_4)\big), \\
		G_3 & := \ker \big( d_{3,-2}^3 \big) \subseteq \ker(\del_3) / \im(\del_4).
	\end{align*}
	However, we also have that $E_{4,-2}^3 \cong H(E_{4,-2}^2) = E_{4,-2}^2 = H_4(\mathcal{D}_4) = \ker(\del_4)$. Since $\ker(\del_4) = 0$, it follows that the differential $\del_{4,-2}^3$ has domain $0$, and is hence the zero map. Therefore $G_2 = 0$, and the result follows.
\end{proof}

We have computed similar short exact sequences in the case where $k=5$, and the proof, omitted, is broadly similar to the above.

\begin{prop}[$k=5$]\label{prop:5-rank_ss}
	Let $\Lambda$ be a row-finite $5$-graph with no sources, and let $\mathcal{D}_5$ be the corresponding chain complex with differentials $\del_1,\ldots , \del_5$ defined in Theorem \ref{thm:gwion}. Let $F_2$, $F_3$ be factors in the ascending filtration of $K_0(\mathcal{A}(\Lambda))$. Then, for some subgroups
	\[
	\begin{array}{ll}
		G_0 \subseteq \coker(\del_1) = H_0(\mathcal{D}_5), & 
		G_4 \cong \ker\big( d_{5,-4}^5\big) \subseteq \ker\big( d_{5,-4}^3\big) \subseteq H_5(\mathcal{D}_5), \\
		G_1 \cong \im\big(d_{5,-4}^5\big) \subseteq H_0(\mathcal{D}_5)/G_0, &
		G_5 \cong \im\big( d_{4,-2}^3 \big) \subseteq H_1(\mathcal{D}_5), \\
		G_2 \cong \im\big( d_{5,-4}^3\big) \subseteq H_2(\mathcal{D}_5), &
		G_6 \cong \ker\big( d_{3,-2}^3\big) \subseteq H_3(\mathcal{D}_5), \\
		G_3 \cong \ker\big( d_{4,-4}^3\big) \subseteq H_4(\mathcal{D}_5), &
	\end{array}
	\]
	there exist short exact sequences as follows:
	\begin{enumerate}[label=(\roman*)]
		\item $ 0 \longrightarrow A:= \dfrac{\coker(\del_1)/G_0}{G_1} \longrightarrow K_0(\mathcal{A}(\Lambda)) \longrightarrow K_0(\mathcal{A}(\Lambda))/A \longrightarrow 0$,
		
		\item $ 0 \longrightarrow A \longrightarrow F_2 \longrightarrow \dfrac{\ker(\del_2)/\im(\del_3)}{G_2} \longrightarrow 0$,
		
		\item $0 \longrightarrow F_2 \longrightarrow K_0(\mathcal{A}(\Lambda)) \longrightarrow G_3 \longrightarrow 0$,
	\end{enumerate}
	and there is an isomorphism $K_1(\mathcal{A}(\Lambda)) \cong F_3 \oplus G_4$, where $F_3$ satisfies
	\[
	0 \longrightarrow \frac{\ker(\del_1)/\im(\del_2)}{G_5} \longrightarrow F_3 \longrightarrow G_6 \longrightarrow 0.
	\]
	\qed
\end{prop}

By the \textit{Kirchberg--Phillips Classification} (\cite{KirPrep}, \cite{Phi2000}), any separable, nuclear, unital, purely-infinite, simple $\cst$-algebra which satisfies the \textit{Rosenberg--Schochet Universal Coefficient Theorem} \cite{RosSch1987} is completely determined by its K-groups and the class of the identity in $K_0$.

\begin{lem}\label{lem:aperiodicity}
	Let $\Gamma$ be a $k$-cube group with adjacency structure $E_1,\ldots , E_k$. Then the induced rank-$k$ graph $\mathcal{G}(\Gamma)$ satisfies the \emph{Aperiodicity Condition} (\cite{KumPas2000}, \cite{MutPrep}).
\end{lem}

\begin{proof}
	The result can be obtained with a slight adaptation to that of Lemma 4.2 in \cite{MutPrep}, as a result of the observations in \S 2 of \cite{RobSte1999}. Briefly, the Aperiodicity Condition is satisfied if, for every vertex $S \in \mathcal{G}(\Gamma)^\mathbf{0}$, there is an infinite path $\varphi \in S \mathcal{G}(\Gamma)^\infty$ such that $\varphi(\mathbf{n} + \mathbf{p}) \neq \varphi(\mathbf{m} + \mathbf{p})$ for all $\mathbf{m}, \mathbf{n}, \mathbf{p} \in \mathbb{N}^k$ with $\mathbf{m} \neq \mathbf{n}$ (c.f. \cite{BNR2014}, \cite{EvaSim2012}, \cite{RobSim2007}). Since $|E_i| \geq 4$ for each $i$, there are always at least three $k$-cubes which are $E_i$-adjacent to any $k$-cube. Moreover, this means that for each $S \in \mathcal{G}(\Gamma)^\mathbf{0}$, we can find two distinct $k$-cubes $T,T' \in \mathcal{G}(\Gamma)^\mathbf{0}$ such that $M_1 M_2 \cdots M_k (S,T) = M_1 M_2 \cdots M_k (S,T') = 1$ Hence, if $\varphi \in S \mathcal{G}(\Gamma)^\infty$ is an infinite path such that there exist $\mathbf{m}, \mathbf{n}$ with $\varphi(\mathbf{n} + \mathbf{p}) = \varphi(\mathbf{m} + \mathbf{p})$ for all $\mathbf{p}$, then we can construct a new aperiodic path by diverting $\varphi$ to a different vertex, breaking this ``infinite cycle'' in the manner of \cite[Lemma 4.2]{MutPrep}.
\end{proof}

\begin{lem}\label{lem:connected}
	Let $\Gamma$ be a $k$-cube group, and let $\mathcal{G}(\Gamma)$ be its induced rank-$k$ graph. Then $\mathcal{G}(\Gamma)$ is strongly connected in the sense that, for any two vertices $S,T \in \mathcal{G}(\Gamma)^\mathbf{0}$, there is a path from $S$ to $T$.
\end{lem}

We give a geometric proof, based on the cube complex $\mathcal{M}(\Gamma)$, although we point out that this can also be proved in the manner of Lemma 4.2 in \cite{KimRob2002}.

\begin{proof}
	Let $\Gamma$ be a $k$-cube group with adjacency structure $E_1,\ldots , E_k$, and consider a pointed $k$-cube $S$ in the cube complex $\mathcal{M}(\Gamma)$. Let $S_H$ be the $k$-cube obtained by reflecting $S$ through the edges labelled by elements of $E_1$, leaving the basepoint and orientation the same as in $S$ (Figure \ref{fig:cube_symmetries}).
	
	Firstly, we aim to show that there is a sequence of $k$-cubes $S = T_0, T_1, \ldots , T_n = S_H$ such that $M_1(T_j,T_{j+1}) = 1$ for all $j$, that is, such that each $k$-cube is $E_1$-adjacent to the next.
	
	Each $k$-cube $X$ in $\mathcal{M}(\Gamma)$ contains two $(k-1)$-faces (that is, $(k-1)$-sub-cubes) labelled by elements of $E_2,\ldots , E_k$. Since the $k$-cubes have a predetermined orientation, we may label these faces $X^L$ and $X^R$, such that $M_1(X,Y) = 1$ if and only if $Y^L = X^R$ and $Y \neq X_H$. We may therefore assign to each $k$-cube $X$ the pair $(X^L,X^R)$ such that, in any sequence $(T_i)$ of adjacent cubes, $T_{i+1}^L = T_i^R$, and $T_{i+1} \neq (T_i)_H$, for all $i$. 
	
	Observe that each $(k-1)$-cube appears as $X^L$ (resp. $X^R$) for some $X \in \mathcal{S}_k$ precisely $|E_1|$ times, and that, by assumption, $|E_1| \geq 4$. Consider $S^R$, and let $T_1$ be a pointed, oriented $k$-cube which is $E_1$-adjacent to $S$; such a cube exists by the above observation. If $T_1^R = S^R$, then $M_1(T_1, S_H) = 1$ and we are done. 
	
	Assume then that $T_1^R \neq S^R$, and let $T_2$ be $E_1$-adjacent to $T_1$. If $T_2^R = S^R$, then $M_1(T_2, S_H) = 1$, and if $T_2^R = T_1^R$, then $M_1(T_2, (T_1)_H) = 1$, and $M_1((T_1)_H , S_H) = 1$. In both cases, we have constructed a sequence of adjacent $k$-cubes linking $S$ to $S_H$.
	
	If $T_q^R = T_p^R$ for any $p < q$, then we obtain the sequence we desire. But also, by the fact that each $(k-1)$-cube appears as $X^R$ for some $X \in \mathcal{S}_k$ an even number of times, there must be some $q > p$ for which $T_q^R = T_p^R$. Hence such a sequence exists, and there is a path connecting the vertices labelled $S$ and $S_H$ in $\mathcal{G}(\Gamma)$.
	
	In the same manner, we may show that there is a sequence of adjacent $k$-cubes connecting each $S \in \mathcal{S}_k$ to each of its symmetries, i.e., the $k$-cubes which belong to the same orbit as $S$ under action by the symmetry group of the $k$-dimensional cube.
	
	Now, we construct the set $\mathcal{P}$ of all $k$-cubes which can be reached by a sequence of adjacent $k$-cubes (in any sequence of directions) from an initial $k$-cube $S$. Certainly $S_H$ is in $\mathcal{P}$, by the above. Moreover, by virtue of Proposition \ref{prop:clique}, $\mathcal{P}$ contains $(|E_1|-1)$-many more \textit{distinct} $k$-cubes which are $E_1$-adjacent to $S$, to total $|E_1|$ distinct $k$-cubes. Each of these $k$-cubes is $E_2$-adjacent to $|E_2|$ $k$-cubes by the same argument. These are distinct from each other by the uniqueness property of \textbf{C3}.
	
	We may proceed inductively to find that $\mathcal{P}$ must contain at least $\prod_{i=1}^k |E_i|$ distinct $k$-cubes, but this is precisely $|\mathcal{S}_k| = |\mathcal{G}(\Gamma)^\mathbf{0}|$. Hence each $k$-cube in $\mathcal{S}_k$ can be reached from some $k$-cube $S$ by a sequence of adjacent $k$-cubes. Equivalently, given any vertex labelled by $S$ in $\mathcal{G}(\Gamma)^\mathbf{0}$, there is a path from $S$ to every other vertex.	
\end{proof}

\begin{thm}\label{thm:determined_by_K}
	Let $\Gamma$ be a $k$-cube group with adjacency structure $E_1,\ldots , E_k$. Then $\mathcal{A}(\Gamma) := \mathcal{A}(\mathcal{G}(\Gamma))$ is separable, nuclear, purely-infinite, simple, and satisfies the Universal Coefficient Theorem. Hence $\mathcal{A}(\Gamma)$ is completely determined by its K-groups and the class of the identity of $\mathcal{A}(\Gamma)$ in $K_0$, up to isomorphism.
\end{thm}

\begin{proof}
	By Lemma \ref{lem:connected} and Proposition 4.8 in \cite{KumPas2000}, it follows that $\mathcal{A}(\Gamma)$ is simple. Also by Lemma \ref{lem:connected}, together with the fact that $|E_i| \geq 4$ for all $i$, it follows that for every $S \in \mathcal{G}(\Gamma)^\mathbf{0}$ we can find $\lambda, \mu \in \mathcal{G}(\Gamma)$ such that $d(\mu) \neq \mathbf{0}$, $r(\lambda) = S$, and $s(\lambda) = r(\mu) = s(\mu)$. Hence by Proposition 4.9 in \cite{KumPas2000} it follows that $\mathcal{A}(\Gamma)$ is purely-infinite.
	
	From Theorem \ref{thm:k-graph} we know that $\mathcal{G}(\Gamma)$ is a row-finite $k$-graph with no sources, and in \cite{Eva2008} it is shown that such a $k$-graph has a corresponding $\cst$-algebra which is separable, nuclear, unital, and satisfies the Universal Coefficient Theorem, hence we are done.
\end{proof}

\begin{prop} \label{prop:ord_of_id}
	Let $\Gamma$ be a $k$-cube group with adjacency structure $E_1, \ldots , E_k$, where $|E_i| = m_i$, and define $\rho := \gcd \lbrace (m_i/2) - 1 \mid 1 \leq i \leq k \rbrace$. Then the order of the class of $\id \in \mathcal{A}(\Gamma)$ in $K_0(\mathcal{A}(\Gamma))$ divides $\rho$.
\end{prop}

\begin{proof}
	The proof of Proposition 5.4 in \cite{KimRob2002} generalises well to the $k$-dimensional cube complexes discussed in this paper.
\end{proof}

\begin{conj}[HK-Conjecture, Matui 2016] \label{matui}
	Let $\mathcal{G}$ be a groupoid which satisfies the conditions of \cite[3.5]{Mat2016}. Then there is an isomorphism
	\[
	K_*(\cst_r(\mathcal{G})) \cong \bigoplus_{p=0}^\infty H_{2p + \epsilon}(\mathcal{G}).
	\]
\end{conj}

\begin{rem}
	Matui's \textit{HK-Conjecture} is known to be false in general (c.f. \cite{Sca2020}), but were it to be verified in the case of $k$-cube groups and their graph $\cst$-algebras, we would be able to refine this result to that of Conjecture \ref{conj:ord_of_id}.
\end{rem}

\begin{conj} \label{conj:ord_of_id}
	Let $\Gamma$ be a $k$-cube group with adjacency structure $E_1, \ldots , E_k$, where $|E_i| = m_i$, and define $\rho := \gcd \lbrace (m_i/2) - 1 \mid 1 \leq i \leq k \rbrace$. Factorise $\rho$ as $2^q r$, where $r$ is an odd number: if $\rho$ is odd then $q = 0$. Then the order of the class of $\id \in \mathcal{A}(\Gamma)$ in $K_0(\mathcal{A}(\Gamma))$ is at most $\rho$, and is:
	\begin{enumerate}[label=(\roman*)]
		\item Equal to $\rho$ if $\rho$ is odd,
		\item Divisible by $\rho / (2^{q})$ if $1 \leq q < (k-1)$,
		\item Divisible by $\rho / (2^{k-1})$ if $q \geq (k-1)$.
	\end{enumerate}
\end{conj}

Given a higher-rank graph $\Lambda$, the sum of all elements of $\mathcal{A}(\Lambda)$ of the form $s_v$, where $v \in \Lambda^{\mathbf{0}}$, is an identity for $\mathcal{A}(\Lambda)$ (e.g. \cite[Remark 3.4]{RobSte1999}). Hence for a $k$-cube group $\Gamma$, the sum	$\sum_{S \in \mathcal{S}(\Gamma)} s_S$ forms an identity in $\mathcal{A}(\Gamma)$. Recall the map
\[
\del_1 : \mathbb{Z}\Gamma^\mathbf{0} \longrightarrow \bigoplus_{i=1}^k \mathbb{Z}\Gamma^\mathbf{0}
\]
defined by the matrix $\big[ I - M_1^T, \ldots , I-M_k^T \big]$. The \textit{Covariance Relation} of \cite[\S 5]{KimRob2002} generalises to $k$-graphs, and so from \cite{RobSte2001} and \textit{Matui's Conjecture} it would follow that the map
\[
\varphi : \coker(\del_1) = \bigg\langle S \in \mathcal{S}_k \biggm\vert \sum_{T \in \mathcal{S}} M_i(S,T) \cdot S \bigg\rangle \longrightarrow K_0(\mathcal{A}(\Gamma))
\] 
which takes $S$ to its class $[S]$ is injective. But each column of $M_i$ has exactly $(m_i - 1)$ ones, the rest of the entries being zero, and so $\Sigma = (m_i - 1)\Sigma$ for each $i \in \lbrace 1, \ldots , k \rbrace$, where $\Sigma := \sum_{S \in \mathcal{S}_k}$. Since $\sum_{S \in \mathcal{S}_k} s_S$ is an identity in $\mathcal{A}(\Gamma)$, the class $[\id] \in K_0$ is the image of $\Sigma$ under $\varphi$; By the above, we also know that $(m_i - 2)\Sigma$ is zero for each $i$.

Write $2 \rho = \gcd\lbrace m_i - 2 \rbrace$, and define the map $\psi : \coker(\del_1) \rightarrow \mathbb{Z}/2\rho$ by $\psi(S) = 1 \mod 2\rho$, as in the proof of \cite[Prop. 5.4]{KimRob2002}. Now,
\[
\prod_i (m_i - 2) = \bigg(\prod_i m_i \bigg) - 2^k \mod 2\rho,
\]
and since $(m_i - 2) = 0 \mod 2\rho$, this means that $\psi(\Sigma) = 2^k \mod 2\rho$, and so $\rho \cdot \psi(\Sigma) = 0 \mod 2\rho$. If $\rho$ is odd, then $\psi(\Sigma)$ has order $\rho$ in $\mathbb{Z}/2\rho$. If $\rho$ is even, then $\rho = 2^q r$ for some odd number $r$, and $\rho \cdot \psi(\Sigma) = 2^{k+q} r \mod (2^{q+1} r)$. Hence the order of $\Sigma$ in $\coker(\del_1)$ is divisible by $\rho$ in the former case, and by $\max\lbrace \rho/(2^q), \rho/(2^{k-1})\rbrace$ in the latter.

\section{Examples for $k=3$ and $k=4$}\label{S:examples}

\begin{ex}\label{ex:3-free2}
	Consider the product of three free groups, each with two generators, defined as follows:
	\[
	\mathbb{F}_2^3 := \langle a_1,a_2,b_1,b_2,c_1,c_2 \mid [a_i,b_j], [a_i,c_j], [b_i,c_j], \text{ for all }i,j \in \lbrace 1,2\rbrace \rangle,
	\]
	where $[x,y]$ denotes the commutator $xyx\inv y\inv$. This is a $3$-cube group with adjacency structure $\lbrace a_i, a_i\inv \rbrace, \lbrace b_i, b_i\inv \rbrace , \lbrace c_i, c_i\inv \rbrace$. We construct the chain complex from Proposition \ref{prop:3-rank_ss} using the three corresponding adjacency matrices, to find that $\coker(\del_1) \cong \mathbb{Z}^8$, and $\ker(\del_2) / \im(\del_3) \cong \ker(\del_1) / \im(\del_2) \cong \mathbb{Z}^{24}$.
	
	Conversely, we can use the \textit{K\"{u}nneth Theorem for tensor products} \cite[\S 9.3]{Weg1993} to calculate the K-theory explicitly up to isomorphism; we have $K_0(\mathcal{A}(\mathbb{F}_2^3)) \cong K_1(\mathcal{A}(\mathbb{F}_2^3)) \cong \mathbb{Z}^{32}$. Then, since all of the groups from Proposition \ref{prop:3-rank_ss} are free Abelian, we are able to deduce that $G_0 = 0$, and $G_1 \cong \mathbb{Z}^8$. This complies with \textit{Matui's HK-Conjecture} \cite{Mat2016}.
\end{ex}

\begin{ex}\label{ex:3-free3}
	Now consider the product $\Gamma$ of three free groups, each with three generators; this is a $3$-cube group whose corresponding cube complex has as universal cover $T(6) \times T(6) \times T(6)$. We again construct the chain complex $\mathcal{D}_3$ from Proposition \ref{prop:3-rank_ss} using the three corresponding adjacency matrices, to find that:
	\begin{itemize}
		\item $\coker(\del_1) \cong \mathbb{Z}^{27} \oplus (\mathbb{Z}/2)^{37}$,
		\item $\ker(\del_2) / \im(\del_3) \cong \mathbb{Z}^{81} \oplus (\mathbb{Z}/2)^{37}$,
		\item $\ker(\del_1) / \im(\del_2) \cong \mathbb{Z}^{81} \oplus (\mathbb{Z}/2)^{74}$,
		\item $\ker(\del_3) \cong \mathbb{Z}^{27}$.
	\end{itemize}
	Hence we have a short exact sequence
	\begin{equation}\label{eq:ses}
		0 \longrightarrow \frac{\mathbb{Z}^{27} \oplus (\mathbb{Z}/2)^{37}}{G_0} \longrightarrow K_0(\mathcal{A}(\Gamma)) \longrightarrow \mathbb{Z}^{81} \oplus (\mathbb{Z}/2)^{37} \longrightarrow 0,
	\end{equation}
	for some $G_0 \subseteq \mathbb{Z}^{27} \oplus (\mathbb{Z}/2)^{37}$, and an isomorphism $K_1(\mathcal{A}(\Gamma)) \cong \mathbb{Z}^r \oplus (\mathbb{Z}/2)^{74}$, where $81 \leq r \leq 108$. We deduce from (\ref{eq:ses}) and the fact that $K_0$ and $K_1$ must have the same torsion-free rank \cite[Prop. 4.1]{Eva2008} that the torsion-free part of $K_0$ is isomorphic to $\mathbb{Z}^r$.
	
	Write $A,B,C$ for the adjacency structure of $\Gamma$. By Lemma \ref{lem:k-1-cube}, the three subgroups of $\Gamma$ isomorphic to $\mathbb{F}_3^2$, obtained by removing one of $A,B,C$ from the generating set, are each $2$-cube groups (or \textit{BM-groups}). The $3$-group $\Gamma$ is a free product with amalgamation of these three groups (Proposition \ref{prop:amalgam}). The K-theory of their induced rank-$k$ graph algebras is given by
	\[
	K_0(\mathbb{F}_3^2) \cong K_1(\mathbb{F}_3^2) \cong \mathbb{Z}^{18} \oplus (\mathbb{Z}/2)^7.
	\]
	Compare this to the K-theory of the rank-$k$ graph algebra induced by $\Gamma$, calculated above to have $K_1(\mathcal{A}(\Gamma)) \cong \mathbb{Z}^r \oplus (\mathbb{Z}/2)^{74}$ and $\rk K_0(\mathcal{A}(\Gamma)) = r$. We discern no obvious structure inherited by the K-theory of $\mathcal{A}(\Gamma)$ from the K-theory induced by its $2$-cube subgroups.
\end{ex}

\begin{ex}\label{ex:26_cubes_pt2}
	Recall the group $\Gamma = \Gamma_{\lbrace 3,5,7\rbrace}$ from Example \ref{ex:26_cubes}. It is readily verifiable that this is a $3$-cube group, and its corresponding cube complex comprises one vertex, $26$ squares labelled by the relators in $R$, and $24$ cubes. We can construct three $192 \times 192$ adjacency matrices $M_1,M_2,M_3$ based on adjacency of pointed cubes in the $\lbrace a_i \rbrace$, $\lbrace b_i \rbrace$, and $\lbrace c_i \rbrace$ directions respectively. We know that $\mathcal{G}(\Gamma)$ is a rank-$3$ graph by Theorem \ref{thm:k-graph}, and so we can input the matrices $M_1,M_2,M_3$ into Proposition \ref{prop:3-rank_ss} to garner information about the algebra $\mathcal{A}(\Gamma)$.
	
	For $\coker(\del_1)$, it suffices to compute the elementary divisors (that is, the diagonal elements of the Smith normal form) of $\del_1$. This is because the cokernel of a linear map is equal to the cokernel of its Smith normal form. The Smith normal form $S(\del_1)$ is a $192 \times 576$ diagonal matrix with entries
	\[
	\underbrace{1, \ldots , 1}_{182 \text{ times}} , 4, 4, 12, \underbrace{0, \ldots , 0}_{7 \text{ times}}.
	\]
	Hence we have $\coker(\del_1) \cong \mathbb{Z}^7 \oplus (\mathbb{Z}/4)^2 \oplus (\mathbb{Z}/12)$, and we are able to work out similarly that $\ker(\del_3) \cong \mathbb{Z}^7$. We also verify using MAGMA that:
	\begin{itemize}
		\item $\ker(\del_2) / \im(\del_3) \cong \mathbb{Z}^{21} \oplus (\mathbb{Z}/4)^2 \oplus (\mathbb{Z}/12)$, and
		\item $\ker(\del_1) / \im(\del_2) \cong \mathbb{Z}^{21} \oplus (\mathbb{Z}/2)^6 \oplus (\mathbb{Z}/ 4)^2 \oplus (\mathbb{Z} / 12)^2$.
	\end{itemize}
	We therefore have a short exact sequence
	\[
	0 \longrightarrow \frac{\mathbb{Z}^7 \oplus (\mathbb{Z}/4)^2 \oplus (\mathbb{Z}/12)}{G_0} \longrightarrow K_0(\mathcal{A}(\Gamma)) \longrightarrow \mathbb{Z}^{21} \oplus (\mathbb{Z}/4)^2 \oplus (\mathbb{Z}/12) \longrightarrow 0,
	\]
	and an isomorphism
	\[
	K_1(\mathcal{A}(\Gamma)) \cong \mathbb{Z}^{21} \oplus (\mathbb{Z}/2)^6 \oplus (\mathbb{Z}/4)^2 \oplus (\mathbb{Z} /12)^2 \oplus G_1,
	\]
	for some $G_0 \subseteq \mathbb{Z}^7 \oplus (\mathbb{Z}/4)^2 \oplus (\mathbb{Z}/12)$ and $G_1 \subseteq \mathbb{Z}^7$. From this and Proposition 4.1 in \cite{Eva2008}, we can deduce that the torsion-free part of $K_0$ is isomorphic to $\mathbb{Z}^r$, and that $K_1 \cong \mathbb{Z}^r \oplus (\mathbb{Z}/2)^6 \oplus (\mathbb{Z}/4)^2 \oplus (\mathbb{Z} /12)^2$, for some $21 \leq r \leq 28$.
	
	We may also compute the cellular homology of the cube complex $\Gamma$; firstly by determining the relevant boundary map matrices. We consider the barycentric subdivision of $\Gamma$, which has the same cellular homology as $\Gamma$, and whose edge set contains no loops. Then, from the Smith normal forms of the boundary maps, we discover that 
	\[
	H_i(\Gamma) \cong 
	\begin{cases}
		\mathbb{Z} & \text{ when } i = 0, \\
		(\mathbb{Z}/2)^2 \oplus (\mathbb{Z}/4)^2 & \text{ when } i = 1, \\
		(\mathbb{Z}/2)^2 \oplus (\mathbb{Z}/12) & \text{ when } i = 2, \\
		\mathbb{Z}^7 & \text{ when } i = 3, \\
		0 & \text{ when } i \geq 4.
	\end{cases}
	\] 
	It is interesting to note the similarities between the cellular homology groups of the cube complex, and the information we have about the K-groups of the $k$-graph algebras.
\end{ex}

\begin{ex}\label{ex:27-cubes}
	Consider the group $\Gamma = \Gamma_{\lbrace 2,3,4 \rbrace}$ from Example 2.36 of \cite{RunStiVdo2019}, defined as follows:
	\[
	\Gamma_{\lbrace 2,3,4\rbrace} := \langle a_1,a_5,a_9,b_2,b_6,b_{10},c_3,c_7,c_{11} \mid R\rangle,
	\]
	where
	\begin{align*}
		R := \big\lbrace &a_1b_2a_5\inv b_{10}\inv , a_1b_6a_9b_{10}, a_1 b_{10} a_9 b_6, a_1b_2\inv a_9\inv b_2\inv ,\\
		&a_1 b_6\inv a_5 b_6\inv , a_1 b_{10}\inv a_5\inv b_2 , a_5 b_2 a_9\inv b_6 , a_5 b_6 a_9\inv b_2 , a_5 b_{10}\inv a_9 b_{10}\inv , \\
		&a_1 c_3 a_5\inv c_3 , a_1c_7a_1\inv c_7\inv , a_1 c_{11} a_9 c_{11} , a_1 c_3\inv a_1 c_{11}\inv ,\\
		&a_5 c_3 a_5 c_7\inv , a_5 c_7 a_9\inv c_7 , a_5 c_{11} a_5\inv c_{11}\inv , a_9 c_3 a_9\inv c_3\inv , a_9 c_7 a_9 c_{11}\inv , \\
		&b_2 c_3 b_6\inv c_{11}\inv , b_2 c_7 b_{10} c_{11}, b_2 c_{11} b_{10} c_7 , b_2 c_3\inv b_{10}\inv c_3\inv ,\\
		&b_2 c_7\inv b_6 c_7\inv , b_2 c_{11}\inv b_6\inv c_3 , b_6 c_3 b_{10}\inv c_7, b_6 c_7 b_{10}\inv c_3 , b_6 c_{11}\inv b_{10} c_{11}\inv \big\rbrace .
	\end{align*}
	This is a $3$-cube group with adjacency structure $\lbrace a_i,a_i\inv\rbrace, \lbrace b_i,b_i\inv \rbrace, \lbrace c_i,c_i\inv\rbrace$, and which acts freely and transitively on $T(6) \times T(6) \times T(6)$, as in Example \ref{ex:3-free3}. The corresponding cube complex $\mathcal{M}$ has one vertex, $27$ squares labelled with the relators in $R$, and $27$ cubes.
	
	We construct the three adjacency matrices, each of which has $216$ rows and columns for the $27$ cubes and each of their eight symmetries (in the manner of Figure \ref{fig:cube_symmetries}). Then, we can use Proposition \ref{prop:3-rank_ss} to reveal information about the K-groups $K_0,K_1$ of the rank-$3$ graph algebra $\mathcal{A}(\Gamma)$.
	
	As in the previous example, we use MAGMA to compute the relevant kernels and cokernels, culminating with:
	\begin{itemize}
		\item $\coker(\del_1) \cong \mathbb{Z}^9 \oplus (\mathbb{Z}/2) \oplus (\mathbb{Z}/20) \oplus (\mathbb{Z}/80)$,
		\item $\ker(\del_3) \cong \mathbb{Z}^9$,
		\item $\ker(\del_2)/\im(\del_3) \cong \mathbb{Z}^{27} \oplus (\mathbb{Z}/2) \oplus (\mathbb{Z}/20) \oplus (\mathbb{Z}/80)$,
		\item $\ker(\del_1) / \im(\del_2) \cong \mathbb{Z}^{27} \oplus (\mathbb{Z}/2)^4 \oplus (\mathbb{Z}/4)^2 \oplus (\mathbb{Z}/8)^2$.
	\end{itemize}
	Then, from Proposition \ref{prop:3-rank_ss} we obtain the short exact sequence
	\[
	0 \rightarrow \frac{\mathbb{Z}^9 \oplus (\mathbb{Z}/2) \oplus (\mathbb{Z}/20) \oplus (\mathbb{Z}/80)}{G_0} \longrightarrow K_0 \longrightarrow \mathbb{Z}^{27} \oplus (\mathbb{Z}/2) \oplus (\mathbb{Z}/20) \oplus (\mathbb{Z}/80) \rightarrow 0,
	\]
	and the isomorphism
	\[
	K_1 \cong \mathbb{Z}^{27} \oplus (\mathbb{Z}/2)^4 \oplus (\mathbb{Z}/4)^2 \oplus (\mathbb{Z}/8)^2 \oplus G_1,
	\]
	where $G_0 \subseteq \mathbb{Z}^9 \oplus (\mathbb{Z}/2) \oplus (\mathbb{Z}/20) \oplus (\mathbb{Z}/80)$ and $G_1 \subseteq \mathbb{Z}^9$. Hence we deduce that the torsion-free part of $K_0$ is isomorphic to $\mathbb{Z}^r$, and $K_1 \cong \mathbb{Z}^r \oplus (\mathbb{Z}/2)^4 \oplus (\mathbb{Z}/4)^2 \oplus (\mathbb{Z}/8)^2$, for some $27 \leq r \leq 36$. The $K_1$ group in particular is distinct from those of Examples \ref{ex:3-free3} and \ref{ex:26_cubes_pt2}, so we may conclude by Theorem \ref{thm:determined_by_K} that the $\cst$-algebras induced by each of the cube complexes are different.
\end{ex}

\begin{rem}
	In each of the examples above, the torsion-free rank $r$ of $K_\ast$ lies within a range of values. If Matui's HK-Conjecture \cite{Mat2016} is true in the case of higher rank graph algebras which arise from $k$-cube groups, then $K_1 \cong H_1(\mathcal{D}_k) \oplus H_3(\mathcal{D}_k)$, and so $r$ must be maximal in this range.
\end{rem}

\begin{ex}
	Consider the product of four free groups, each with two generators, defined as follows:
	\begin{multline*}
		\mathbb{F}_2^4 := \langle a_1,a_2,b_1,b_2,c_1,c_2,d_1,d_2 \mid [a_i,b_j], [a_i,c_j], [a_i,d_j], \\ [b_i,c_j], [b_i,d_j], [c_i,d_j] \text{ for all }i,j \in \lbrace 1,2\rbrace \rangle.
	\end{multline*}
	This is a $4$-cube group with adjacency structure $\lbrace a_i, a_i\inv \rbrace, \ldots , \lbrace d_i, d_i\inv \rbrace$. We construct the chain complex from Proposition \ref{prop:4-rank_ss} using the four corresponding adjacency matrices, and obtain short exact sequences:
	\begin{enumerate}[label=(\roman*)]
		\item $0 \longrightarrow \mathbb{Z}^{16}/G_0 \longrightarrow K_0(\mathcal{A}(\mathbb{F}_2^4)) \longrightarrow \dfrac{K_0(\mathcal{A}(\mathbb{F}_2^4))}{\mathbb{Z}^{16}/G_0} \longrightarrow 0$,
		
		\item $0 \longrightarrow \mathbb{Z}^{16}/G_0 \longrightarrow F_2 \longrightarrow \mathbb{Z}^{96} \longrightarrow 0$,
		
		\item $0 \longrightarrow F_2 \longrightarrow K_0(\mathcal{A}(\mathbb{F}_2^4)) \longrightarrow G_1 \longrightarrow 0$,
		
		\item $0 \longrightarrow \mathbb{Z}^{64}/G_2 \longrightarrow K_1(\mathcal{A}(\mathbb{F}_2^4)) \longrightarrow G_3 \longrightarrow 0$,
	\end{enumerate}
	for some subgroups $G_0, G_1 \subseteq \mathbb{Z}^{16}$, and $G_2,G_3 \subseteq \mathbb{Z}^{64}$. Compare this with the values obtained for $K_0$ and $K_1$ by the K\"{u}nneth Theorem for tensor products, a result we are able to use due to the groups $\coker(\del_1)$ and $\ker(\del_4)$ being torsion free. It gives $K_0(\mathcal{A}(\mathbb{F}_2^4)) \cong K_1(\mathcal{A}(\mathbb{F}_2^4)) \cong \mathbb{Z}^{128}$.
	
	Each of the groups in the sequences above is free Abelian, so we can use this information to deduce from (iv) that $G_2 = 0$, $G_3 \cong \mathbb{Z}^{64}$, and $G_0 \oplus G_1 \cong \mathbb{Z}^{16}$.
\end{ex}

\begin{ex}
	Consider the group $\Gamma = \Gamma_{\lbrace 1,2,3,4 \rbrace}$, found as a result of \cite{RunStiVdo2019}, and defined as follows:
	\[
	\Gamma_{\lbrace 1,2,3,4\rbrace} := \langle a_1,a_2,a_3,b_1,b_2,b_3,c_1,c_2,c_3,d_1,d_2,d_3 \mid R\rangle,
	\]
	where
	\begin{align*}
		R := \big\lbrace &a_1 b_1 a_3\inv b_1, a_1 b_1\inv a_2\inv b_3, a_1 b_2 a_2 b_2, a_1 b_2\inv a_3 b_3\inv, \\
		&a_1 b_3 a_2\inv b_1\inv, a_1 b_3\inv a_3 b_2\inv, a_2 b_3 a_3 b_3, a_3 b_1 a_2\inv b_2, a_3 b_2 a_2\inv b_1 \\
		&a_1 c_1 a_2\inv c_1, a_1 c_1\inv a_1 c_3\inv, a_1 c_2 a_1\inv c_2\inv, a_1 c_3 a_3 c_3, \\
		&a_2 c_1 a_2 c_2\inv, a_2 c_2 a_3\inv c_2, a_2 c_3 a_2\inv c_3\inv, a_3 c_1\inv a_3\inv c_1, a_3 c_2 a_3 c_3\inv \\	
		&a_1 d_1 a_3\inv d_3, a_1 d_1\inv a_2 d_2, a_1 d_2 a_2 d_1\inv, a_1 d_2\inv a_1 d_3\inv, \\
		&a_1 d_3 a_3\inv d_1, a_2 d_1 a_2 d_3\inv, a_2 d_2\inv a_3 d_3, a_2 d_3 a_3 d_2\inv, a_3 d_1 a_3 d_2 \\	
		&b_1 c_1 b_3\inv c_1, b_1 c_1\inv b_2\inv c_3, b_1 c_2 b_2 c_2, b_1 c_2\inv b_3 c_3\inv, \\
		&b_1 c_3 b_2\inv c_1\inv, b_1 c_3\inv b_3 c_2\inv, b_2 c_3 b_3 c_3, b_3 c_1 b_2\inv c_2, b_3 c_2 b_2\inv c_1, \\	
		&b_1 d_1 b_2\inv d_1, b_1 d_1\inv b_1 d_3\inv, b_1 d_2 b_1\inv d_2\inv, b_1 d_3 b_3 d_3, \\
		&b_2 d_1 b_2 d_2\inv, b_2 d_2 b_3\inv d_2, b_2 d_3 b_2\inv d_3\inv, b_3 d_1\inv b_3\inv d_1, b_3 d_2 b_3 d_3\inv \\	
		&c_1 d_1 c_3\inv d_1, c_1 d_1\inv c_2\inv d_3, c_1 d_2 c_2 d_2, c_1 d_2\inv c_3 d_3\inv, \\
		&c_1 d_3 c_2\inv d_1\inv, c_1 d_3\inv c_3 d_2\inv, c_2 d_3 c_3 d_3, c_3 d_1 c_2\inv d_2, c_3 d_2 c_2\inv d_1 \big\rbrace .
	\end{align*}
	We have written a program in Python which determines whether a group is a $4$-cube group, and if so, outputs four adjacency matrices. In this example, $\Gamma$ is a $4$-cube group, but the adjacency matrices are very large. MAGMA has a limit on the dimensions of the input, so we are currently exploring other languages and ways around the bounds of the software.
	
	This example underlines the complexity in computing K-groups of $k$-graph algebras beyond $k=3$, though we think it is nevertheless interesting to have a source of higher-rank graphs of arbitrary rank to draw from, which arise in a concrete way.
\end{ex}

\section{Higher-rank graphs arising as double covers of cube complexes}\label{S:two_vertices}

In this section, we deduce information about the K-theory of a certain class of rank-$k$ graphs with two vertices, which arise as double covers of the cube complexes discussed above. Whereas the rank-$k$ graphs $\mathcal{G}(\Gamma)$ had vertices labelled by the $k$-cubes of $\Gamma$, these rank-$k$ graphs $\Lambda$ have vertices labelled by the vertices $(v,0), (v,1)$ of the double cover of the complex $\mathcal{M}(\Gamma)$, where $v$ is the single vertex of $\mathcal{M}(\Gamma)$. For further detail on how these graphs arise, we direct the reader to \cite[\S 8]{LawSimVdoPrep}, but we outline the process here.

Recall the cube complex $\mathcal{M}(\Gamma)$, which has one vertex $v$, and has as cover the product of $k$ trees by Proposition \ref{prop:prod_of_trees} and Example \ref{ex:M}. Let $\ell : \Gamma \rightarrow \mathbb{Z}/2$ be a labelling of the elements of $\Gamma$. Then we obtain a cover $\tilde{\mathcal{M}}^2$ of $\mathcal{M}(\Gamma)$ with vertex set $\lbrace v \rbrace \times (\mathbb{Z}/2)$, and edge set $\Gamma \times (\mathbb{Z}/2)$. In $\tilde{\mathcal{M}}^2$, for a given element $a \in \Gamma$, either $(a,0)$ and $(a,1)$ are loops based at $(v,0)$, $(v,1)$ respectively, or $(a,0)$ goes from $(v,1)$ to $(v,0)$ and $(a,1)$ from $(v,0)$ to $(v,1)$.

We therefore have the following construction. For some non-negative integers $m_1,\ldots , m_k$ and $n_1,\ldots , n_k$, not necessarily distinct, consider the matrices
\[
D_i := \begin{bmatrix}
	2m_i & 0 \\ 0 & 2m_i
\end{bmatrix}, \qquad
T_i := \begin{bmatrix}
	0 & 2n_i \\ 2n_i & 0
\end{bmatrix}.
\]
For each $i \in \lbrace 1,\ldots , k\rbrace$, we let $M_i$ equal one of $D_i$ or $T_i$. Provided that we have at least one $T_i$, we can construct a rank-$k$ graph $\Lambda$ with our choices of $M_i$ as incidence matrices (Figure \ref{fig:2-verts}). Such rank-$k$ graphs are clearly cofinal and satisfy the Aperiodicity Condition, so their graph $\cst$-algebras are uniquely determined by their K-theory.

Our aim is to apply Evans' result (Proposition \ref{prop:3-rank_ss} above) to compute the K-theory of the corresponding rank-$3$ graph algebras $\Lambda$ in each case, for all $m_i, n_i \geq 2$. To do this, we must investigate the Smith normal forms of the boundary maps $\del_i$. We do so firstly for $k=3$, in order to illustrate the method for arbitrary $k$.

\begin{lem}\label{lem:DT_3}
	Let $k=3$, and let $\Lambda$ be one of the three possible rank-$3$ graphs constructed as above, namely one with incidence matrices
	\begin{enumerate}[label=(\alph*)]
		\item $M_1 = T_1$, $M_2 = D_2$, $M_3 = D_3$,
		\item $M_1 = T_1$, $M_2 = T_2$, $M_3 = D_3$,
		\item $M_1 = T_1$, $M_2 = T_2$, $M_3 = T_3$,
	\end{enumerate}
	up to reordering. Write
	\[
	a_i := \begin{cases}
		1 - 2m_i & \text{ whenever } M_i = D_i,\\
		1 - 4n_i^2 & \text{ whenever } M_i = T_i.
	\end{cases}
	\]
	Recall the matrices $\del_1,\del_2,\del_3$ from Proposition \ref{prop:3-rank_ss}, and write $I_2, 0_2$ to denote the $2 \times 2$ identity and zero matrix, respectively.  Then the Smith normal forms of the $\del_i$ are given by
	\[
	S(\del_1) = S(\del_3)^T = \begin{bmatrix}
		1&0&0&0&0&0 \\ 0&g&0&0&0&0
	\end{bmatrix}, \quad \text{and}\quad
	S(\del_2) = \begin{bmatrix}
		I_2 & 0_2 & 0_2 \\
		0_2 & g I_2 & 0_2 \\
		0_2 & 0_2 & 0_2
	\end{bmatrix},
	\] 
	where $g := \gcd(a_1,a_2,a_3)$.
\end{lem}

\begin{figure}[h]
	\begin{center}
		\includegraphics[scale=0.7]{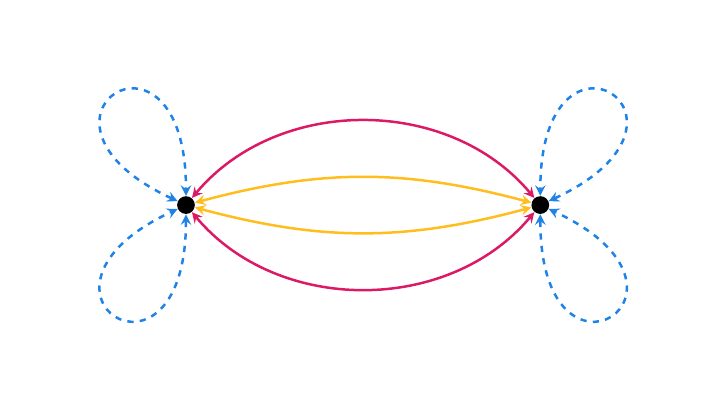}
	\end{center}
	\caption[.]{The ($1$-skeleton of the) $3$-graph with two vertices and incidence matrices $\big[ \begin{smallmatrix}
			4 & 0 \\ 0 & 4
		\end{smallmatrix} \big]$, 
		$\big[ \begin{smallmatrix}
			0 & 2 \\ 2 & 0
		\end{smallmatrix} \big]$, and 
		$\big[ \begin{smallmatrix}
			0 & 2 \\ 2 & 0
		\end{smallmatrix} \big]$, respectively represented by blue, magenta, and yellow arrows.}\label{fig:2-verts} 
\end{figure}

The proof relies on an argument based on the \textit{Cauchy--Binet Theorem}: for a matrix $\del$ of rank $r$, the product of the invariant factors of the Smith normal form $S(\del)$ is equal to the greatest common divisor of all of the determinants of the $r \times r$ minors of $\del$.

\begin{proof}	
	Consider the matrices $\del_1, \del_2, \del_3$ in each of the three cases.
	In case (a), we have
	\[
	\del_1 = \begin{bmatrix}
		1&-2n_1&1-2m_2&0&1-2m_3&0 \\ -2n_1&1&0&1-2m_2&0&1-2m_3
	\end{bmatrix}.
	\]	
	It is known that $\del_1$ and $\del_3^T$ have the same Smith normal form, so it suffices to check $S(\del_1)$. The matrix $\del_1$ has rank $2$, and the non-zero determinants of its $2 \times 2$ minors are given by
	\[
	\begin{gathered}
		1-4n_1^2,\,
		(1-2m_2)^2,\,
		(1-2m_3)^2,\,
		1-2m_2,\, \\
		1-2m_3,\,
		2n_1(1-2m_2),\,
		2n_1(1-2m_3),\,
		(1-2m_2)(1-2m_3).
	\end{gathered}
	\]
	Let $g$ be the greatest common divisor of these determinants. Then, after some relatively harmless algebra, we find that $g = \gcd(1-4n_1^2, 1-2m_2, m_2-m_3)$.
	
	For $\del_2$, which has rank $4$, we compute the determinants of all $4 \times 4$ minors:
	\[
	\begin{gathered}
		(1-4n_1^2)^2, \,
		(1-2n_2)^4, \,
		(1-2m_3)^4, \, \\
		(1-4n_1^2)(1-2n_2)^2, \,
		(1-4n_1^2)(1-2m_3)^2, \,
		(1-2n_2)^2(1-2m_3)^2,
	\end{gathered}
	\]
	and find that the greatest common divisor of these is equal to $g^2$. Hence 
	\[
	S(\del_1) = S(\del_3)^T = \begin{bmatrix}
		1&0&0&0&0&0 \\ 0&g&0&0&0&0
	\end{bmatrix}, \quad \text{and}\quad
	S(\del_2) = \begin{bmatrix}
		I_2 & 0_2 & 0_2 \\
		0_2 & g I_2 & 0_2 \\
		0_2 & 0_2 & 0_2
	\end{bmatrix},
	\]
	where $I_2, 0_2$ are the $2 \times 2$ identity matrix and zero matrix, respectively. Cases (b) and (c) are shown in a similar manner.
\end{proof}

\begin{prop}\label{prop:DT_3}
	Let $k=3$, let $\Lambda$ be one of the three possible rank-$3$ graphs constructed as above, and let $g$ be the corresponding value as defined in Lemma \ref{lem:DT_3}. Then
	\begin{enumerate}[label=(\roman*)]
		\item If $g=1$, then $K_0(\mathcal{A}(\Lambda)) \cong K_1(\mathcal{A}(\Lambda)) \cong 0$,
		\item If $g \geq 2$, then $K_1(\mathcal{A}(\Lambda)) \cong (\mathbb{Z}/g) \oplus (\mathbb{Z}/g)$, and $K_0(\mathcal{A}(\Lambda))$ is isomorphic to a group of order $g^2$.
	\end{enumerate}
\end{prop}	

\begin{proof}
	If $g=1$, then $\coker(\del_1) \cong \ker(\del_3) \cong 0$, and we may apply Corollary \ref{cor:3-rank}(i). Using the Smith normal form to calculate the homologies, we find that $K_0(\mathcal{A}(\Lambda)) \cong K_1(\mathcal{A}(\Lambda)) \cong 0$.
	
	If $g \geq 2$, then $\coker(\del_1) \cong (\mathbb{Z} / g)$, and $\ker(\del_3) \cong 0$. Once again, we apply Proposition \ref{prop:3-rank_ss} to obtain the short exact sequence
	\[
	0 \longrightarrow (\mathbb{Z}/g) \longrightarrow K_0(\mathcal{A}(\Lambda)) \longrightarrow \ker(\del_2)/\im(\del_1) \longrightarrow 0,
	\]
	and the isomorphism $K_1(\mathcal{A}(\Lambda)) \cong \ker(\del_1) / \im(\del_2)$. It is well known (consult, for example, \cite{DHSW2003}) that, when the differentials are finitely-generated, the homologies of a chain complex can be computed via the formula
	\[
	\ker(\del_i)/\im(\del_{i+1}) \cong \mathbb{Z}^{c-r-s} \oplus \bigoplus_{j=1}^r (\mathbb{Z} / a_j) ,
	\]
	where $c$ is the number of columns of $\del_i$, $r := \rk(\del_{i+1})$, $s := \rk(\del_i)$, and $a_j$ are the non-zero entries of $S(\del_{i+1})$. In our example, this gives
	\[
	\ker(\del_1) / \im(\del_2) \cong (\mathbb{Z} / g) \oplus (\mathbb{Z} / g) , \qquad \text{and} \qquad \ker(\del_2) / \im(\del_3) \cong (\mathbb{Z} / g),
	\]
	and the result follows.
\end{proof}

\begin{prop}\label{prop:DT}
	Let $\Lambda$ be one of the $k$ possible rank-$k$ graphs constructed as above, with incidence matrices $M_1, \ldots , M_k$. Write 
	\[
	a_i := \begin{cases}
		1 - 2m_i & \text{ whenever } M_i = D_i,\\
		1 - 4n_i^2 & \text{ whenever } M_i = T_i,
	\end{cases}
	\]
	and $g := \gcd(a_1,\ldots , a_k)$. Recall the matrices $\del_1 , \ldots , \del_k$ from Theorem \ref{thm:gwion}, and write $I_s$ to denote the $s \times s$ identity matrix. Then the Smith normal forms of the matrices $\del_i$ are given by
	\[
	S(\del_1) = S(\del_k)^T = \begin{bmatrix}
		1&0&\mathbf{0} \\ 0&g&\mathbf{0}
	\end{bmatrix}_{2 \times 2k},
	\]
	and for $2 \leq i \leq k-1$, the rank of $\del_i$ is equal to $2R := 2{k-1 \choose i-1}$, and
	\[
	S(\del_i) = \begin{bmatrix}
		I_R & 0 & \mathbf{0} \\
		0 & gI_R & \mathbf{0} \\
		\mathbf{0} & \mathbf{0} & \mathbf{0}
	\end{bmatrix}_{2 {k \choose i-1} \times 2 {k \choose i}} .
	\]
\end{prop}

\begin{proof}
	The main argument of the proof is that, knowing the non-trivial invariant factor of $S(\del_1) = S(\del_k)^T$, along with the rank of each $\del_i$, suffices to completely determine the Smith normal form of these matrices.
	
	Firstly, we compute $S(\del_1) = S(\del_k)^T$ in the same manner as Lemma \ref{lem:DT_3}. Secondly, we determine the shape (and hence the rank) of the matrices $\del_2, \ldots , \del_{k-1}$. Recall from Theorem \ref{thm:gwion} the sets
	\[
	N_l := \begin{cases}
		\big\lbrace \bmu := (\mu_1 , \ldots , \mu_l) \in \lbrace 1, \ldots , k\rbrace^l \mid \mu_1 < \cdots < \mu_l \big\rbrace &\text{if } l \in \lbrace 1,\ldots , k\rbrace, \\
		\lbrace * \rbrace & \text{if } l=0, \\
		\emptyset & \text{otherwise}.
	\end{cases}
	\]
	From now on, we treat $\del_i$ exclusively as block matrices, with shape $|N_{i-1}| \times |N_i|$, and which comprise $2 \times 2$ blocks which are either zero or $\pm (I_2-M_j^T)$ for $j \in \lbrace 1,\ldots , k\rbrace$. Notice that $N_i$ has $(\begin{smallmatrix} k \\ i \end{smallmatrix})$ elements, each a strictly-increasing tuple of length $i$ whose entries belong to $\lbrace 1 , \ldots , k \rbrace$. Thus we consider $N_i$ and $N_{i-1}$ as ordered sets, with $\bmu > \bmu'$ if there is some $r$ such that $\mu_r > \mu_r'$ and $\mu_t = \mu_t'$ for $t < r$. We label the (blocks of the) rows and columns of each matrix $\del_i$ with the elements, in order, of $N_{i-1}$ and $N_i$ respectively, writing $\bmu(p) \in N_{i-1}$ for the label of row $p$, and $\bnu(q) \in N_i$ for the label of column $q$.
	
	Denote by $\del_i(p,q)$ a block of $\del_i$ in position $(p,q)$. Then $\del_i(p,q)$ is either $\pm (I_2-M_j^T)$ for some $j$, or the $2 \times 2$ zero matrix. It is non-zero if and only if $\bmu(p)$ can be obtained by deleting one element, say $\nu(q)_t$, of $\bnu(q)$. Furthermore, the value of $\nu(q)_t$ determines the index of the matrix $M_j$, and the position $t$ determines the sign of $\del_i(p,q)$ as follows:
	\[
	\del_i(p,q) = \begin{cases}
		I_2-M_j^T & \text{if $t$ is odd,} \\
		M_j^T - I_2 & \text{if $t$ is even.}
	\end{cases}
	\]
	Using this interpretation of the blocks of $\del_i$, we conclude that the number of non-zero blocks is equal to $|N_{i-1}| \cdot (k-(i-1)) = |N_i| \cdot (k-i)$.
	
	It remains to deduce the rank of each matrix $\del_i$. Firstly, observe that
	\[\rk(\del_i) \leq \min 2 \bigg\lbrace {k \choose i-1} , {k \choose i} \bigg\rbrace.
	\]
	By construction of $\del_i$, there exists a $2(\begin{smallmatrix} k-1 \\ i-1 \end{smallmatrix}) \times
	2(\begin{smallmatrix} k-1 \\ i-1 \end{smallmatrix})$ minor of $\del_i$, with block-elements $I_2 - M_1^T$ on the diagonal, obtained by considering those column-blocks labelled by elements $\bnu \in N_i$ with $\nu_1 = 1$, and row-blocks labelled by elements $\bmu \in N_{i-1}$ with $\mu_1 \neq 1$. There are
	\[
	{k \choose i} - {k-1 \choose i} = {k-1 \choose i-1}, \quad \text{and} \quad
	{k \choose i-1} - {k-1 \choose i-2} = {k-1 \choose i-1}
	\]
	such column- and row-blocks, respectively, and so
	\[
	\rk(\del_i) \geq 2 {k-1 \choose i-1}.
	\]
	We claim that this is in fact an equality, and we display here an outline of the proof. Suppose that there exists a non-vanishing minor $A$ which is larger than the one above. Given the number of non-zero blocks, and the fact that minors are square matrices, the diagonal (or when relevant, the anti-diagonal) of $A$ must not contain any zero blocks, else $A$ vanishes. Using elementary row operations, we take the last $(\begin{smallmatrix} k-1 \\ i-1 \end{smallmatrix})$ row-blocks of $\del_i$ and move them to the top, such that the elements $\bmu \in N_{i-1}$ with $\mu_1 = 1$ now label the rows at the bottom of $\del_i$. Hence we must be able to find a column labelled by an element $\bnu \in N_i$ such that $\nu_1 \neq 1$, and from which we can find the $\bmu$, but this is a contradiction.
	
	Indeed, by some careful considerations, it turns out that regardless of the row operations performed, we will arrive at such a contradiction. Thus $\rk(\del_i) = 2{k-1 \choose i-1}$, and from this the result is readily verified.
\end{proof}

%

\begin{prop}[$k=4$]\label{prop:DT_4}
	Let $\Lambda$ be one of the four possible rank-$4$ graphs constructed as above, and let $g$ be the corresponding value as defined in Lemma \ref{prop:DT}. Then:
	\begin{enumerate}[label=(\roman*)]
		\item If $g=1$, then $K_0(\mathcal{A}(\Lambda)) \cong K_1(\mathcal{A}(\Lambda)) \cong 0$,
		\item If $g \geq 2$, then we can find short exact sequences
		\begin{enumerate}[label=(\alph*)]
			\item $0 \longrightarrow (\mathbb{Z}/g)/G_0 \longrightarrow K_0(\mathcal{A}(\Lambda)) \longrightarrow (\mathbb{Z}/g)^3 \longrightarrow 0$,
			\item $0 \longrightarrow (\mathbb{Z}/g)^3/G_2 \longrightarrow K_1(\mathcal{A}(\Lambda)) \longrightarrow G_3 \subseteq (\mathbb{Z}/g) \longrightarrow 0$,
		\end{enumerate}
		where $G_0$ and $G_2$ are as in Proposition \ref{prop:4-rank_ss}.
	\end{enumerate}
\end{prop}

\begin{proof}
	Firstly, if $g=1$, it follows that $\ker(\del_4) = 0$ and we may apply Corollary \ref{cor:4-rank}(i). But $H_3(\mathcal{D}_4)$, $H_2(\mathcal{D}_4)$, and $H_1(\mathcal{D}_4)$ are all trivial, and so the K-theory of $\mathcal{A}(\Lambda)$ is trivial.
	
	If $g \geq 2$, then $\ker(\del_4) = 0$, $\coker(\del_1) \cong (\mathbb{Z} / g)$, and Proposition \ref{prop:4-rank_ss} gives us:
	\begin{enumerate}[label=(\alph*)]
		\item $0 \longrightarrow \coker(\del_1) / G_0 \longrightarrow K_0(\mathcal{A}(\Lambda)) \longrightarrow \ker(\del_2) / \im(\del_3) \longrightarrow 0$,
		\item $0 \longrightarrow \dfrac{\ker(\del_1) / \im(\del_2)}{G_2} \longrightarrow K_1(\mathcal{A}(\Lambda)) \longrightarrow G_3 \longrightarrow 0$.
	\end{enumerate}
	Again, we use the Smith normal forms to compute the homologies $\ker(\del_i) / \im(\del_{i+1})$ of the chain complex $\mathcal{D}_4$, which reveal that
	\[
	\ker(\del_1)/\im(\del_2) \cong \ker(\del_2) / \im(\del_3) \cong (\mathbb{Z} / g)^3, \qquad  \text{and} \qquad \ker(\del_3) / \im(\del_4) \cong (\mathbb{Z} / g),
	\]
	and the above sequences reduce to those we desire.
\end{proof}

\nocite{KonVdo2015}
\nocite{CunKri1980}

\section*{Acknowledgments} 
The authors would like to thank David Evans and Gwion Evans for discussing the connections between the theory of buildings, higher-rank graphs, and the K-theory of their corresponding $\cst$-algebras at the Newton Institute, Cambridge in Spring 2017.

The authors express their gratitude to Newcastle University for providing an excellent research environment, and to the EPSRC for funding part of this research.

\bibliographystyle{amsplain}
\bibliography{k=3bib}


\end{document}